\documentclass[12pt,a4paper,twoside,english]{amsart}

\usepackage[english]{babel}
\usepackage[latin1]{inputenc}
\usepackage{amssymb, amsmath, amsrefs,amsthm}
\usepackage{graphicx}
\usepackage{hyperref}
\usepackage{enumitem}
\usepackage{color}

\makeatletter
\g@addto@macro\bfseries{\boldmath}
\makeatother

\newcommand{\norm}[1]{\left\Vert#1\right\Vert}
\newcommand{\brkt}[1]{\left(#1\right)}

\newcommand{\abs}[1]{\left|#1\right|}

\newcommand{\esc}[1]{\langle #1\rangle}

\renewcommand{\Bbb}{\mathbb} 
\newcommand{\J}{{J}}
\newtheorem{teo}{Theorem}[section]
\newtheorem{cor}[teo]{Corollary}
\newtheorem{exam}[teo]{Example}

\newtheorem{lem}[teo]{Lemma}
\newtheorem{defi}[teo]{Definition}
\newtheorem{prop}[teo]{Proposition}
\newtheorem{rem}[teo]{Remark}

\usepackage{float}

\newcommand{\R}{\mathbb{R}}

\newcommand{\N}{\mathbb{N}}
\newcommand{\SW}{\mathcal{S}}
\newcommand{\BMO}{\mathrm{BMO}}
\newcommand{\bmo}{\mathrm{bmo}}
\newcommand{\ddd}{\,\text{\rm{\mbox{\dj}}}}
\newcommand{\dd}{\mathrm{d}}

\newcommand{\Rw}{\mathrm{w}}
\newcommand{\Symb}{\sigma}

\title[] {Some endpoint estimates for bilinear Coifman-Meyer multipliers}
\author[S. Arias]{Sergi Arias}
\address{Department of Mathematics, Stockholm University, SE-106 91 Stockholm, Sweden}
\email{arias@math.su.se}
\author[S.~Rodr\'iguez-L\'opez]{Salvador Rodr\'iguez-L\'opez}
\address{Department of Mathematics, Stockholm University, SE-106 91 Stockholm, Sweden}
\email{s.rodriguez-lopez@math.su.se}

\thanks{The second author is partially supported by the Spanish Government grant MTM2016-75196-P}

\makeatletter
\@namedef{subjclassname@2020}{%
  \textup{2020} Mathematics Subject Classification}
\makeatother


\begin{document}

\begin{abstract}
    In this paper we establish mapping properties of bilinear Coifman-Meyer multipliers acting on the product spaces $H^1(\R^n)\times\bmo(\R^n)$ and $L^p(\R^n)\times\bmo(\R^n)$, with $1<p<\infty$. As application of these results, we obtain some related Kato-Ponce-type inequalities involving the endpoint space $\bmo(\R^n)$, and we also study the pointwise product of a function in $\bmo(\R^n)$ with functions in $H^1(\R^n)$, $h^1(\R^n)$ and $L^p(\R^n)$, with $1<p<\infty$.
\end{abstract}
\maketitle

\section{Introduction}
The aim of this paper is to obtain some endpoint estimates for bilinear Coifman-Meyer multipliers, which are operators of the form
\begin{equation*}
    T_\Symb(f,g)(x):=\iint\Symb(\xi,\eta)\widehat{f}(\xi)\widehat{g}(\eta)e^{ix(\xi+\eta)}\ddd\xi\ddd\eta,
\end{equation*}
where $\ddd \xi=(2\pi)^{-n}\dd \xi$ denotes the normalised Lebesgue measure, and the symbol $\Symb$ is a smooth function on $\R^n\times\R^n\setminus\lbrace(0,0
)\rbrace$ satisfying that, for all pair of multi-indices $\alpha,\beta\in\N^n$
\begin{equation}\label{CM_mult_property}
    \abs{\partial_\xi^\alpha\partial_\eta^\beta\Symb(\xi,\eta)}\lesssim(\vert\xi\vert+\vert\eta\vert)^{-\vert\alpha\vert-\vert\beta\vert}, \quad \mbox{for all $(\xi,\eta)\neq(0,0)$.}
\end{equation}

The literature on mapping properties of bilinear Coifman-Meyer multipliers is vast, so we conceal our discussion to those works more relevant to the current paper. The study of those operators was initiated by R. Coifman and Y. Meyer \cite{CMbook}, and can be seen as particular instances of bilinear Calder\'on-Zygmund operators. As such, it follows from the general theory by  L. Grafakos and R. Torres \cite{Grafakos-Torres}, that these operators map $L^p(\R^n)\times L^q(\R^n)$ continuously into $L^r(\R^n)$ under the relation $1/r=1/p+1/q$, where $1<p,q\leq\infty$ and $1/2<r<\infty$. Moreover, and relevant to this paper, the following endpoint estimates are also known:
\begin{enumerate}[label=(\roman*)]
    \item\label{Obj:1} $H^1(\R^n)\times L^\infty(\R^n) \to L^1(\R^n)$;
    \item\label{Obj:2}  $L^p(\R^n)\times L^\infty(\R^n) \to L^p(\R^n)$, for $1<p<\infty$;
    \item 
    $L^\infty(\R^n)\times L^\infty(\R^n) \to \BMO(\R^n)$.
\end{enumerate}
Mapping properties of bilinear Coifman-Meyer multipliers acting on $\bmo(\R^n)\times \bmo(\R^n)$ were established in \cite{Paper1}*{Theorem 7.1}. Nevertheless, to the best of our knowledge, the boundedness properties of such operators acting on product spaces of the type $L^p(\R^m)\times \bmo(\R^n)$, for $1<p<\infty$, and $H^1(\R^m)\times \bmo(\R^n)$ is not covered by the results existing in the literature. 

The aim of this paper is to study these remaining cases. In this way, our investigation is a continuation of that initiated in \cite{Paper1}.
These results are obtained as a consequence (see Corollary \ref{Cor:main_bmo}) of a more general theorem (see Theorem \ref{CM_main_thm} below), where we establish boundedness properties analogous to \ref{Obj:1} and \ref{Obj:2}, where $L^\infty(\R^n)$ is replaced by a suitable intermediate space $X_w(\R^n)$, lying in between $L^\infty(\R^n)$ and $\bmo(\R^n)$, associated to an admissible weight $w$. The adequate choice of weight, allows both to recover the existing boundeness properties \ref{Obj:1} and \ref{Obj:2}, as well as to obtain those involving $\bmo(\R^n)$. The precise definition of admissible weights and that of the space $X_w(\R^n)$ are given in Section \ref{sect:admissible_w} below. 

The range of Coifman-Meyer multipliers acting on $\bmo(\R^n)\times \bmo(\R^n)$ obtained in \cite{Paper1}, lies in a space of generalised smoothness, defined as a potential-type space, involving $\BMO(\R^n)$, and denoted by $\J_{\Rw_1}(\BMO(\R^n))$ in this paper (see Section \ref{generalised_smoothness_spaces} for its definition). This space is analogous to those introduced by R. S. Strichartz \cite{Strichartz}, $\J^s(\BMO)$ with $s\in \R$, where $\J^s=(1-\Delta)^{s/2}$ denotes the Bessel potential operator. Heuristically, the space $\J^s(\BMO)$ consists of  \lq\lq derivatives of order $s$" of functions in $\BMO(\R^n)$, while the space $\J_{\Rw_1}(\BMO(\R^n))$ consists of \lq\lq derivatives of logarithmic order" of functions in $\BMO(\R^n)$. The formal definition and properties of these spaces of generalised smoothness are given in Section \ref{generalised_smoothness_spaces} below.

The literature on the study of spaces of generalised smoothness, and those of logarithmic smoothness in particular, is considerably large. Among those we refer, for their relevance to this paper, to the works of V. Mikhailets and A. A. Murach \cite{Mikhailets}, A. Caetano and S. Moura \cite{Caetano-Moura} and S. Moura \cite{Moura}. Nevertheless, we also refer the interested reader to the recent work of O. Dominguez and S. Tikhonov \cite{dominguez2018function} and the references therein, for a recent overview of the field and related results for function spaces with logarithmic smoothness. 

As in \cite{Paper1}, to obtain the mapping properties studied in this paper, we introduce a family of spaces of generalised smoothness, of potential-type, associated to either $L^p(\R^n)$ or $H^1(\R^n)$. We also show (see Proposition \ref{pot_as_triebel} and Proposition \ref{pot_as_hor} below) that, in several cases, these spaces coincide with some spaces already existing in the literature \cites{Caetano-Moura, Moura, Mikhailets}.

The study of bilinear Coifman-Meyer multipliers can be reduced to that of paraproducts of the type
\begin{equation*} %
    \Pi(f,g)=\int_0^\infty(Q_tf)(P_tg)m(t)\frac{\dd t}{t},
\end{equation*}
where $Q_t$ and $P_t$ are frequency localisation operators and $m$ is a bounded function (see Section \ref{SecPara}). In this way, our main results are obtained by the use of new estimates for bilinear paraproducts (see Theorem \ref{bmo_Lp}).

One of the applications that follows from the main theorem of this paper is an endpoint inequality of Kato-Ponce-type. These kind of inequalities have been largely studied, and we used as a main reference the works by L. Grafakos and S. Oh \cite{Grafakos-Oh}, by L. Grafakos, D. Maldonado and V. Naibo \cite{Gra-Mal_Nai}, V. Naibo and A. Thomson \cite{Naibo-Thomson} and K. Koezuka and N. Tomita \cite{Koezuka_Tomita}. In particular, it is known (see e.g. \cite{Koezuka_Tomita}*{Corollary 1.2} that 
\begin{equation*}
    \norm{J^{s}(fg)}_{h^p(\R^n)}\lesssim\norm{J^{s}f}_{h^{p}(\R^n)}\norm{g}_{L^{\infty}(\R^n)}+\norm{f}_{h^{p}(\R^n)}\norm{J^{s}g}_{L^{\infty}(\R^n)}
\end{equation*}
provided $s>0$, $1\leq p<\infty$. We obtain related end-point estimates to this one (see Corollary \ref{Kato-Ponce} below), with $L^\infty(\R^n)$ replaced by $\bmo(\R^n)$, albeit with the stronger restriction on $s$ to be larger than $4n+1$.

Regarding the endpoint case $s=0$, namely estimates for the product of two functions, one in $\BMO(\R^n)$ and the other in the Hardy space $H^1(\R^n)$, have been investigated  by A. Bonami, T. Iwaniec, P. Jones and M. Zinsmeister in \cite{Bon-Iwa-Jon-Zin} and by A. Bonami, S. Grellier and L. D. Ky in \cite{Bon-Gre-Ky}. Likewise, J. Cao, L. D. Ky and D. Yang  \cite{Cao-Ky-Yang} studied the counterpart problem where one of the terms lies in the local Hardy space $h^p(\R^n)$, for $0<p\leq 1$, and the other in the local $\bmo(\R^n)$ space. In these studies, the product is decomposed in two terms, one belonging to  $L^1(\R^n)$, and the other in a suitable Musielak-Orlicz-Hardy space.  

As a final application of our main results, we obtain some counterparts of these results for the range $1\leq p<\infty$. In both Corollary \ref{cor:Lp_potential_prod} and  Corollary \ref{teo_prod}, the product is realised in one of the potential-type spaces of generalised smoothness introduced in this paper. In addition, Corollary  \ref{teo_prod_2} allows the product to be decomposed in two terms, one belonging to  $L^1(\R^n)$, and the other in a Musielak-Orlicz-Hardy space.     

This article is organised as follows. Section \ref{prel} is devoted to introduce the notions and tools needed to state and prove the main results. It is divided in three parts. In a first one, we introduce some function spaces and technical results related to them. In a second one, we introduce the notion of admissible weights, and the associated function spaces used throughout the paper.
In the last part, we define the potential-type spaces of generalised smoothness, and study some of their properties.

In Section \ref{main_results} we state the main theorem of the article. The mapping properties for paraproducts (Theorem \ref{bmo_Lp}) are presented in Section \ref{SecPara}. The proof of Theorem \ref{CM_main_thm} is given in Section \ref{CFM_section} and some of its consequences are introduced in Section \ref{Applications}. 

Finally, we present an Appendix with a self-contained and direct proof of the $L^2$-estimates obtained for paraproducts in Theorem \ref{bmo_Lp} by interpolation. 

\section{Preliminaries}\label{prel}

\subsection{Generalities}
The notation $A\lesssim B$ will be used to indicate the existence of a constant $C>0$ such that $A\leq C B$. Similarly, we will write $A\thickapprox B$ if both $A\lesssim B$ and $B\lesssim A$ hold. We will also use the notation $\langle\xi\rangle=(1+\vert\xi\vert^2)^{1/2}$ for $\xi\in\R^n$.

The space of Schwartz functions will be denoted by $\SW(\R^n)$ and its topological dual, the space of tempered distributions, by $\SW'(\R^n)$. For a function $f\in \SW(\R^n)$ we define its Fourier transform as
\[
    \widehat{f}(\xi)=\int_{\R^n}f(x)e^{-i x\xi}\ddd x
\]
and we will write
\[
    a(tD)f(x)=\int_{\R^n}a(t\xi)\widehat{f}(\xi)e^{i x\xi}\ddd\xi
\]
for appropriate symbols $a$, or simply $a(D)$ when $t=1$.

A diverse collection of function spaces will be appearing throughout the paper. Lets recall that the Hardy space $H^1(\R^n)$, is the space of tempered distributions $f$ for which the non-tangential maximal function
\begin{equation}\label{hardy_def}
    x\mapsto\sup_{t>0}\sup_{\abs{x-y}<t}\abs{(\Phi_t\ast f)(y)}
\end{equation}
belongs to $L^1(\R^n)$, endowed with the norm
\[
    \norm{f}_{H^1(\R^n)}:=\norm{\sup_{t>0}\sup_{\abs{x-y}<t}\abs{(\Phi_t\ast f)(y)}}_{L^1(\R^n)}.
\]
Here $\Phi$ is a Schwartz function with $\int\Phi=1$ and the notation $\Phi_t(x)=t^{-n}\Phi(x/t)$ will be used from now on, with $t>0$ and $x\in\R^n$. 

The local version of the Hardy space, denoted by $h^1(\R^n)$ and introduced by D. Goldberg \cite{Goldberg}, is the space of tempered distributions $f$ for which the truncated non-tangential maximal function
\[
    x\mapsto\sup_{0<t<\frac{1}{2}}\sup_{\abs{x-y}<t}\abs{(\Phi_t\ast f)(y)}
\]
belongs to $L^1(\R^n)$, endowed with the norm given by 
\begin{equation}\label{def:norm_h1}
    \norm{f}_{h^1(\R^n)}:=\norm{\sup_{0<t<\frac{1}{2}}\sup_{\abs{x-y}<t}\abs{(\Phi_t\ast f)(y)}}_{L^1(\R^n)}.
\end{equation} 

The space of functions of bounded mean oscillation, $\BMO(\R^n)$, which can be identified with the dual space of $H^1(\R^n)$, is the set of all those locally integrable functions $f$ on $\R^n$ for which
\[
    \norm{f}_{\BMO(\R^n)}:=\sup_Q\frac{1}{\vert Q\vert}\int_Q\vert f(x)-f_Q\vert\dd x<\infty.
\]
The supremum is taken over all cubes in $\R^n$ whose sides are parallel to the axis, while $\vert Q\vert$ denotes the Lebesgue measure of the cube $Q$ and $f_Q$ is the average of $f$ over $Q$, namely $f_Q=\frac{1}{\vert Q\vert}\int_Qf(x)\dd x$.

The local version of $\BMO(\R^n)$ was considered by D. Goldberg in \cite{Goldberg}, and it will be denoted by $\bmo(\R^n)$. It is defined to be the set of all locally integrable functions $f$ on $\R^n$ for which
\begin{equation}\label{Zipi}
    \norm{f}_{\bmo(\R^n)}:=\sup_{\ell(Q)<1}\frac{1}{\vert Q\vert}\int_Q\vert f(x)-f_Q\vert\dd x + \sup_{\ell(Q)\geq 1}\frac{1}{\vert Q\vert}\int_Q\vert f(x)\vert\dd x<\infty.
\end{equation}
Here $\ell(Q)$ denotes the side length of the cube $Q$. The function space $\bmo(\R^n)$ is the dual space of $h^1(\R^n)$ and it is continuously embedded in $\BMO(\R^n)$.

We recall (see e.g. \cite{Grafakos2}*{Section 3.3.1}) that a measure $\dd\mu(x,t)$ on $\R^{n+1}_+$ is a Carleson measure if there exists a constant $C>0$ such that
\begin{equation}\label{ecu33}
\mu(T(Q))\leq C\abs{Q}
\end{equation}
for all cube $Q$ in $\R^n$, where $T(Q):=Q\times (0,\ell(Q)]$. The norm of the Carleson measure, denoted by $\norm{\mu}_{\mathcal{C}}$, is considered to be the infimum of the set of all constants $C>0$ satisfying \eqref{ecu33}. 

We shall also recall the relations between Carleson measures and functions in $\BMO(\R^n)$, given in the the following result, whose proof can be found in \cite{Grafakos2}*{Theorem 3.3.8 b)} and \cite{Grafakos2}*{Theorem 3.3.8 c)} respectively.
\begin{teo}\label{rem_BMO_CM}
    Let $\Psi$ be a Schwartz function satisfying $\widehat{\Psi}(0)=0$ and
    \[
    \sup_{\xi\in\R^n}\int_0^\infty\abs{\widehat{\Psi}(t\xi)}^2\frac{\dd t}{t}<\infty.
    \]
    Then, for all function $b\in\BMO(\R^n)$, the measure defined by
    \[
    \dd\mu(x,t)=\abs{(\Psi_t\ast b)(x)}^2\dd x\frac{\dd t}{t}
    \]
    is a Carleson measure  with norm bounded by a constant times $\norm{b}_{\BMO(\R^n)}^2$.
\end{teo}
\begin{teo}\label{IO_BMO_CM}
    Let $K_t(x,y)$, $t>0$, be a collection of functions defined on $\R^n\times\R^n$ for which there exists $\delta>0$ such that
    \[
        \abs{K_t(x,y)}\lesssim\frac{t^\delta}{(t+\abs{x-y})^{n+\delta}}
    \]
    for all $t>0$ and $x,y\in\R^n$. Define for every $t>0$ the linear operators
    \[
        R_tf(x)=\int_{\R^n}K_t(x,y)f(y)\dd y.
    \]
    Assume that $R_t1\equiv 0$ for all $t>0$ and that the estimate
    \[
        \int_0^\infty\!\!\!\!\int_{\R^n}\abs{R_tf(x)}^2\dd x\frac{\dd t}{t}\lesssim\norm{f}_{L^2(\R^n)}^2
    \]
    holds for all $f\in L^2(\R^n)$. Then, for all $b\in\BMO(\R^n)$, the measure defined by
    \[
        \abs{R_tb(x)}^2\dd x\frac{\dd t}{t}
    \]
    is a Carleson measure with norm bounded by a constant times $\norm{b}_{\BMO(\R^n)}^2$.
\end{teo}

We shall also recall the following classical result of L. Carleson (see \cite{stein2}*{p. 236} or \cite{Grafakos2}*{Corollary 3.3.6}).

\begin{teo}\label{Feff-Ste}
    Let $\mu(x,t)$ be a Carleson measure on $\R^{n+1}_+$ and $0<p<\infty$. For every $\mu$-measurable function $F(x,t)$ on $\R^{n+1}_+$ the estimate
    \[
    \int_{\R^{n+1}_+}\abs{F(x,t)}^p\dd\mu(x,t)\lesssim\norm{\mu}_{\mathcal{C}}\int_{\R^n}(F^*(x))^p\dd x
    \]
    holds, where $F^*$ denotes the nontangential maximal function
    \[
    F^*(x):=\sup_{t>0}\sup_{\abs{x-y}<t} F(y,t), \quad x\in\R^n.
    \]
\end{teo}

We shall also need the following technical result 
shown in  \cite{Paper2}*{Proposition 4.11}. 
\begin{prop}\label{hardy_CM}
    Let $f\in H^1(\R^n)$ and $v(t,x)\in L^\infty((0,\infty)\times\R^n)$. Let $G(t,x)$ be a measurable function on $\R^{n+1}_+$, and assume that the measure defined by
    \[
    \dd\mu_G(t,x):=\abs{G(t,x)}^2\frac{\dd t}{t}\dd x
    \]
    is a Carleson measure with norm $\norm{\dd\mu_G}_{\mathcal{C}}$. Then we have that
    \[
    \abs{\int\!\!\!\!\int_0^\infty(Q_tf)(x)G(t,x)v(t,x)\frac{\dd t}{t}\dd x}\lesssim\norm{f}_{H^1(\R^n)}\norm{\dd\mu_G}_{\mathcal{C}}^{1/2}\norm{v}_{L^\infty((0,\infty)\times\R^n)}.
    \]
\end{prop}

\subsection{Admissible weights}\label{sect:admissible_w}
The endpoint results obtained in this article involve some function spaces introduced in \cite{Paper1}. We shall first recall their definition (see \cite{Paper1}*{Definition 4.2}), and for convenience, gather together in Proposition \ref{prop:gather} below, some of their properties discussed in that paper.

\begin{defi}\label{X_w}
    Let $w:(0,\infty)\rightarrow (0,\infty)$ be a weight function satisfying the following properties:
        \begin{enumerate}[label=\Roman*)]
            \item \label{I}For every {compact} interval $I\subseteq(0,\infty)$ we have that
                  \[
                    0<\inf_{t\in I}\left(\inf_{s>0}\frac{w(st)}{w(s)}\right)\leq\sup_{t\in I}\left(\sup_{s>0}\frac{w(st)}{w(s)}\right)<\infty;
                  \]
            \item \label{II} There exists $N>0$ such that $\sup_{t>0}w(t)(1+1/t)^{-N}<\infty$;
            \item \label{III} $\inf_{t>0} w(t)>0$.
        \end{enumerate}
    Let $\phi$ be a Schwartz function with frequency support inside a ball centred at the origin, and set $P_tf:=\widehat{\phi}(tD)f$. Then $X_w(\R^n)$ is defined to be the set of all locally integrable functions for which
    \[
        \norm{f}_{X_w(\R^n)}:=\norm{f}_{\BMO(\R^n)}+\sup_{t>0}\frac{\norm{P_tf}_{L^\infty(\R^n)}}{w(t)}<\infty.
    \]
\end{defi}
\begin{prop}\label{prop:gather} Let $w$ be a weight function satisfying \ref{I},\ref{II} and \ref{III}, and a function $\phi$ as above. 
\begin{enumerate}
    \item The definition of the space $X_w(\R^n)$ does not depend on the different choices of function $\phi$, in the sense that different choices induce equivalent norms.
    \item The embeddings $L^\infty(\R^n)\subset X_w(\R^n)\subset \bmo(\R^n)$ hold.
    \item If $w\approx 1$, then $X_w(\R^n)=L^\infty(\R^n)$ with equivalent norms.
    \item For $w(t)=1+\log_+ (1/t)$, we have that $X_w(\R^n)=\bmo(\R^n)$ with equivalent norms.
\end{enumerate}
\end{prop}

\begin{rem} 
        Given a weight satisfying \ref{I},  \ref{II} and \ref{III} above, and given any positive constant $c>0$, the weight $cw$ satisfies the same properties, and $X_w(\R^n)=X_{cw}(\R^n)$ with equivalence of norms. So, without loss of generality, multiplying $w$ by a constant, one can  assume that $w\geq 1$. This is how the condition \ref{III} is stated in \cite{Paper1}*{Definition 4.2}. Moreover, the class of weights satisfying these conditions were called admissible in that paper. Here, we will reserve that terminology for those weights defined below. 

    Note also that, as a direct consequence of \ref{I},  \ref{II} and \ref{III}, $w$ also satisfies that for all $0<c_1\leq c_2$, there exist $0<d_1\leq d_2$ such that
    \begin{equation}
         \label{equivalence}
         c_1\leq \frac{t}{s}\leq c_2\quad \mbox{implies that}\quad  d_1\leq \frac{w(t)}{w(s)}\leq d_2. \end{equation}
\end{rem}

The following definition is a minor modification of that in the paper of A. Caetano and S. Moura \cite{Caetano-Moura}*{Definition 2.1}, that we shall adopt hereafter in this paper. 
\begin{defi}%
    Let $w:(0,1]\rightarrow (0,\infty)$ be a monotonic function, and extend it to $w:(0,\infty)\to (0,\infty)$ by defining $w(t)=w(1)$ for all $t\geq 1$. We say that $w$ is an \textit{admissible} weight if it satisfies that there exist $c,d>0$ such that for all $j\geq 0$ 
    \begin{equation}\label{log:doubling}
        cw(2^{-j})\leq  w(2^{-2j}) \leq d w(2^{-j}).
    \end{equation}
\end{defi}
\begin{exam}\label{exe-logb}
Example of admissible weights are those functions of the form 
    \begin{equation*}
      w_b(t)=\left(1+\log_+(1/t)\right)^b, \quad b\in \R.
    \end{equation*} 
    
    Indeed, it was observed in \cite{Caetano-Moura}*{Example 2.2} that weights defined on $(0,1]$ by an expression of the form $\abs{\log cx}^b$ are admissible, provided $c\in (0,1]$. Note now that for $t\in (0,1)$, we can write 
    \[
        \abs{\log(t/e)}^b=(1+\log(1/t))^b.
    \]
\end{exam}

\begin{lem} Let $w$ be an admissible weight and let  $\Theta:(0,\infty)\to (0,\infty)$ be a monotonic function satisfying \eqref{equivalence}. Then $\Theta(w(t))$ is also an admissible weight. In particular, $w^{-1}$, and $cw$ for all constant $c>0$ are also admissible. 
\end{lem}
\begin{proof}
Note that the monotonicity of $w$ and $\Theta$, yields the monotonocity of the composition. 
Moreover, \eqref{log:doubling} and \eqref{equivalence}, imply that $\Phi(w(t))$ also satisfies \eqref{log:doubling}.

The last part of the statement easily follows from the first part by taking $\Theta(t)$ equal to $t^{-1}$, and $ct$ respectively.
\end{proof}

\begin{rem}\label{remloglog} Using the previous lemma and the example above, one can construct other admissible weights such as 
\[
    w(t):=(1+\log_+(1/t))^{b_1}\brkt{1+\log (1+\log_+(1/t))}^{b_2},
\]
with $b_1\cdot b_2\geq 0$. 
\end{rem}
\begin{lem}
    Let $w$ be an admissible weight. Then $w$ satisfies \ref{I} and \ref{II} in Definition \ref{X_w} above. Moreover, condition \ref{III} holds for an admissible weight $w$ if, and only if, $w$ is either non-increasing, or satisfying that for all $t>0$, $w(t)\approx 1$.
\end{lem}
\begin{proof} 
We shall provide a proof of the first part of the statement for $w$ being non-increasing. The non-decreasing case is treated analogously.

We know that, by \cite{Caetano-Moura}*{Lemma 2.3}, there exists $b\geq 0$ such that for all $0<t\leq 1$ it holds that \begin{equation*}%
        1\leq \inf_{0<s\leq 1}\frac{w(ts)}{w(s)}\leq \frac{w(t)}{w(1)}\leq \sup_{0<s\leq 1}\frac{w(ts)}{w(s)}\lesssim(1+\abs{\log t})^b.
    \end{equation*}
    Since $w$ is constant on $[1,\infty)$, these inequalities imply
    \[
        1\lesssim w(t)\lesssim(1+\log_+ 1/t)^b,
    \]
    from where $w$ satisfies \ref{II}. 

    Using these two inequalities, and a change of variables in the case that $t\geq 1$, we deduce that for all $t>0$
   \[
          (1+\abs{\log t})^{-b}\lesssim\inf_{s>0}\frac{w(ts)}{w(s)}\leq\sup_{s>0}\frac{w(ts)}{w(s)}\lesssim(1+\abs{\log t})^b.
   \]
     Now, since $w$ is non-increasing, it follows that 
     \[
        \sup_{t>1} \sup_{s>0} \frac{w(st)}{w(s)}\leq 1\leq 
        \inf_{t<1} \inf_{s>0} \frac{w(st)}{w(s)}.
     \]
     These imply that
    \[
       (1+\log_+ t)^{-b}\lesssim \inf_{s>0}\frac{w(ts)}{w(s)}\leq \sup_{s>0}\frac{w(ts)}{w(s)}\lesssim (1+\log_+1/t)^b, 
    \]
    which yields that $w$ satisfies \ref{I}.
    
    In the non-increasing case, the condition $\inf_t w(t)>0$ is equivalent to say that for all $t>0$, $w(t)\geq w(1)>0$, which implies that $w$ satisfies \ref{III}. If $w$ is non-decreasing, the condition $\inf_{t>0} w(t)=w(0^+)>0$, is equivalent to the property that $w(t)\approx 1$ for all $t>0$.
\end{proof}

\begin{rem}
   Observe that if a non-decreasing admissible weight $w$ satisfies $\inf_{t>0} w(t)=w(0^+)>0$ then one has that $w(t)\approx 1$. Hence, by Proposition \ref{prop:gather}, the space $X_w(\R^n)$ in  Definition \ref{X_w} coincides with $L^\infty(\R^n)$. 
\end{rem}

\subsection{Generalised smoothness-type spaces}\label{generalised_smoothness_spaces}
The rest of this section is devoted to the definition and main properties of the spaces of generalised smoothness that appear in the main result of this paper.

To this end, let $\varphi_0$ be a positive and radially monotonically decreasing Schwartz function, supported in $\lbrace\abs{x}\leq 3/2\rbrace$, which is identically one on $\lbrace\abs{x}\leq 1\rbrace$. We define then $\varphi_1(x):=\varphi_0(x/2)-\varphi_0(x)$ and $\varphi_j(x):=\varphi_1(2^{-j+1}x)$ for $x\in\R^n$ and $j$ any integer bigger than one. In particular, it holds that $\sum_{j=0}^\infty\varphi_j(x)=1$ for all $x\in\R^n$, and the collection of functions $(\varphi_j)_{j\geq 0}$ forms a resolution of the unity.

The following definition can be implicitly found in \cite{Caetano-Moura}.

\begin{defi}\label{Regularised_weight}
Let $w$ be an admissible weight, and let  $(\varphi_j)_{j\geq 0}$ be a resolution of unity as above. We say that the function 
\begin{equation}\label{log_symbol}
    \Rw(\xi)=\sum_{j=0}^\infty w(2^{-j})\varphi_j(\xi),\quad \xi\in\R^n,
\end{equation}
 is the regularisation of $w$ (associated to the resolution of unity $(\varphi_j)_{j\geq 0}$).
\end{defi}

It was shown in \cite{Caetano-Moura}*{Lemma 3.1}, that both $\Rw$ and $1/\Rw$ are smooth functions on $\R^n$ such that for all multi-index $\alpha\in\N^n$, and for all $\xi\in \R^n$
\begin{equation}\label{ecu1sym}
    \abs{(\partial^\alpha\Rw)(\xi)}\lesssim w(1/\langle\xi\rangle)\langle\xi\rangle^{-\abs{\alpha}}
\end{equation}
and
\begin{equation}\label{ecu2sym}
    \abs{\left(\partial^\alpha\left(\frac{1}{\Rw}\right)\right)(\xi)}\lesssim \frac{1}{w(1/\langle\xi\rangle)}\langle\xi\rangle^{-\abs{\alpha}}.
\end{equation}
 In addition, using \eqref{log_symbol} and the fact that $w(1/\abs{\xi})=w(1)$ for all $\abs{\xi}\leq 1$ we obtain the estimates
\begin{equation}\label{sigma_equiv_weight}
    \Rw(\xi)\approx w(1/\abs{\xi})\approx w(1/\langle\xi\rangle),\quad \xi\neq 0.
\end{equation}
This motivates the terminology of regularisation, as $\Rw$ is smooth and also essentially encodes all the point-wise information of $w$.

\begin{defi}
    For $1\leq p\leq \infty$, let $X^p$ denote either $L^p(\R^n)$ if $1<p<\infty$, and $X^1=H^1(\R^n)$ or $X^\infty=\BMO(\R^n)$ in the case $p=1$ and $p=\infty$ respectively.
\end{defi}

\begin{lem}\label{lem:KN1}
Let $\sigma\in \mathcal{C^\infty}(\R^n)$ belonging to the Kohn-Nirenberg class $S^0(\R^n)$. That is, it satisfies that for all multi-index $\alpha\in\N^n$ 
    \[
        \sup_{\xi\in\R^n} \esc{\xi}^{\abs{\alpha}}\abs{\partial^\alpha_\xi\Symb(\xi)}<+\infty.
    \]
Then, for all $1\leq p\leq \infty$, $\sigma(D):X^p\to X^p$
is a bounded operator.
\end{lem}
\begin{proof}
The boundedness on $L^p(\R^n)$ for $1<p<\infty$ can be found in \cite{Grafakos1}*{Theorem 6.2.7} and that on $H^1(\R^n)$ can be found in \cite{Cuerva-Francia}*{Theorem III.7.30}). Duality and self adjointness, yield the boundedness on $\BMO(\R^n)$. 
\end{proof}

\begin{defi} Let $w$ be an admissible weight, and let $\Rw$ denote its regularisation given by \eqref{log_symbol}. We define the linear operators
\begin{equation*}%
\J_{\Rw}f:=\Rw(D)f \quad \mathrm{and} \quad \J_{\Rw^{-1}}f:=\Rw^{-1}(D)f,
\end{equation*}
defined initially for $f\in\SW(\R^n)$. 
\end{defi}

\begin{rem} In the last statement, and hereafter, we write $\Rw^{-1}(\xi)$ to denote the function $1/\Rw(\xi)$.
\end{rem}
\begin{prop}\label{prop:221} Let $w$ be an admissible weight, and let $\Rw$ be its  regularisation. 
\begin{enumerate}
    \item The operators $\J_{\Rw}$ and $\J_{\Rw^{-1}}$ are linear and continuous on $\SW(\R^n)$ and $\SW'(\R^n)$, being each other inverses.
    \item If $\inf_{t>0} w(t)>0$, then  $\Rw^{-1}\in S^0(\R^n)$, and so, for all $1\leq p\leq \infty$, 
    \[
        \J_{\Rw^{-1}}:X^p\to X^p
    \]
    is a bounded operator.
\end{enumerate}
\end{prop}
\begin{proof}
Both operators are continuous on $\SW(\R^n)$ since $\Rw$, $1/\Rw$ are smooth and all their derivatives have at most polynomial growth. This yields that both $\J_{\Rw}$ and $\J_{\Rw^{-1}}$ are continuous on $\SW'(\R^n)$.  In addition, by the commutativity of Fourier multipliers it  follows that $\J_{\Rw}$ and $\J_{\Rw^{-1}}$ are each other inverses.

The assumption $\inf_{t>0} w(t)>0$ and \eqref{ecu2sym} imply that the symbol $\Rw^{-1}\in S^0(\R^n)$, and so the boundedness is a direct consequence of the previous lemma.
\end{proof}

\begin{defi}%

Let $w$ be an admissible weight and let $\Rw$ be its regularisation given by \eqref{log_symbol}. We define the space
\begin{equation*}%
    \J_{\Rw}(X^p):=\lbrace f\in\SW'(\R^n):\J_{\Rw^{-1}} f\in X^p\rbrace,
\end{equation*}
equipped with the norm
\[
    \norm{f}_{\J_{\Rw}(X^p)}:=\norm{\J_{\Rw^{-1}} f}_{X^p}.
\]
Similarly one defines 
\begin{equation*}
    \J_{\Rw^{-1}}(X^p):=\lbrace f\in\SW'(\R^n):\J_{\Rw} f\in L^p(\R^n)\rbrace,
\end{equation*}
equipped with the norm
\[
    \norm{f}_{\J_{\Rw^{-1}}(X^p)}:=\norm{\J_{\Rw} f}_{X^p}.
\]
\end{defi}

\begin{prop} Let $1\leq p\leq \infty$, and let $w$ be an admissible weight. The definition of the spaces $\J_{\Rw}(X^p)$ and $\J_{\Rw^{-1}}(X^p)$ is independent of the resolution of the unity chosen to regularise $w$.
\end{prop}
\begin{proof}
Let $({\varphi_j^1})_j$ and $({\varphi_j^2})_j$ be two resolutions of the unity, and let $\Rw_{1}$ and $\Rw_{2}$ be the corresponding regularisations of $w$. 

For all multi-index $\alpha$, the Leibniz rule, \eqref{ecu1sym} and \eqref{ecu2sym} yield
\[
    \abs{\partial^\alpha \brkt{\Rw_2\Rw_1^{-1}}}\lesssim \frac{w(1/\langle\xi\rangle)}{w(1/\langle\xi\rangle)} \esc{\xi}^{\alpha}\approx \esc{\xi}^{\alpha},
\]
and so $\Rw_2\Rw_1^{-1}\in S^0(\R^n)$. Analogously, one shows that $\Rw_2^{-1}\Rw_1\in S^0(\R^n)$.

Let us prove that $\J_{\Rw_1}(X^p)=\J_{\Rw_2}(X^p)$, by showing that their defining norms are equivalent. By the symmetry of the problem it is indeed enough to prove that $\J_{\Rw_2}(X^p)\subset \J_{\Rw_1}(X^p)$. Lemma \ref{lem:KN1} and the commutativity of Fourier multipliers yield 
\[
     \norm{\J_{\Rw_1^{-1}} f}_{X^p}= \norm{(\Rw_1^{-1}\Rw_2)(D) \Rw_2^{-1}(D) f}_{X^p}\lesssim \norm{\J_{\Rw_2^{-1}} f}_{X^p}
\]
for $f\in \J_{\Rw_2}(X^p)$, as we wanted to show. 

The independence of the definition of $\J_{\Rw^{-1}}(X^p)$ on the resolution of the identity is obtained analogously. So we omit the details.
\end{proof}
\begin{prop}\label{prop:inverse_w} Let $w$ be an admissible weight, and define the admissible weight $u=w^{-1}$. Let $\mathrm{w}$ and $\mathrm{u}$ denote respectively their  regularisation given by \eqref{log_symbol}. Then it holds that for all $1\leq p\leq \infty$, $\J_{\Rw}(X^p)=\J_{\mathrm{u}^{-1}}(X^p)$ with equivalent norms. Namely
\begin{equation*}
        \norm{\J_\mathrm{u} f}_{L^p(\R^n)}\approx\norm{\J_{\mathrm{w}^{-1}}f}_{L^p(\R^n)},
\end{equation*}
for all $f\in \J_{\Rw}(X^p)=\J_{\mathrm{u}^{-1}}(X^p)$
\end{prop}
\begin{proof}
With an argument similar to the one on the previous proposition, one obtains that both $\mathrm{u}\Rw$ and $(\mathrm{u}\Rw)^{-1}$ belong to $S^0(\R^n)$.

The proof runs similarly, and it is a consequence of Lemma \ref{lem:KN1}, so the details are left to the reader.
\end{proof}

\begin{prop}
Let $1\leq p\leq \infty$, and let $w$ be an admissible weight.
\begin{enumerate}
    \item For all $1\leq p<\infty$, the space $\J_{\Rw}(X^p)$ endowed with the norm $\norm{.}_{\J_{\Rw}(X^p)}$ is a Banach space. In the case $p=\infty$ we have that $\norm{.}_{\J_{\Rw}(\BMO(\R^n))}$ defines a norm in $\J_{\Rw}(\BMO(\R^n))$ after identifying functions which differ almost everywhere by a constant, making it a Banach space;
    \item The dual space of $\J_{\Rw}(X^p)$ is $\J_{\Rw^{-1}}(X^{p'})$ when $1\leq p<\infty$, where $p'$ denotes the H\"older conjugate exponent of $p$;
    \item If $\inf_{t>0} w(t)>0$, then $X^p$ is continuously embedded in $\J_{\Rw}(X^p)$. 
\end{enumerate}
\end{prop}
\begin{proof}
    The fact that $\norm{.}_{\J_{\Rw}(X^p)}$ defines a norm follows from $\left(X^p,\norm{.}_{X^p}\right)$ being a normed space and the linearity of $\J_{\Rw^{-1}}$, when $1\leq p<\infty$.
    
    In the endpoint case $p=\infty$,  the space of functions of bounded mean oscillation defines a normed space, provided that functions which differ almost everywhere by a constant are identified. As a consequence, if $\norm{f}_{\J_{\Rw}(\BMO(\R^n))}=0$ then $f=\Rw(0)C$ in $\SW'(\R^n)$ for some constant function $C$. Hence the same identification is needed on $\J_{\Rw}(\BMO(\R^n))$.
    
    Let $1\leq p\leq\infty$. To show the completeness of  $\J_{\Rw}(X^p)$, let $(f_n)_{n\in\N}$ be a Cauchy sequence in $\J_{\Rw}(X^p)$. Then $(J_{\Rw^{-1}}f_n)_{n\in\N}$ is a Cauchy sequence in $X^p$ and, since $X^p$ is complete, we can find $g\in X^p$ for which $\J_{\Rw^{-1}}f_n\rightarrow g$ in $X^p$ as $n\rightarrow\infty$. Then $J_\Rw g$ belongs to $\J_\Rw(X^p)$ and the sequence $(f_n)_{n\in\N}$ converges to $\J_\Rw g$ in $\J_\Rw(X^p)$. 
    
    To show the second statement, notice that $\Lambda\in (\J_{\Rw}(X^p))^*$ is equivalent to $\Lambda\circ \J_{\Rw}\in (X^p)^*$, which is identified with $X^{p'}$. This implies the existence of a unique $g_\Lambda\in X^{p'}$ such that for all $f\in X^p$
    \[
        \Lambda ( \J_{\Rw}f)=  \esc{g_\Lambda,f}=\esc{\J_{\Rw^{-1}}g_\Lambda,J_\Rw f},
    \]
    which yields the existence of a unique $h_\Lambda\in \J_{\Rw^{-1}}(X^{p'})$, representing $\Lambda$, from where we obtain $(\J_{\Rw}(X^p))^*\subset \J_{\Rw^{-1}}(X^{p'})$. The other inclusion is obtained analogously.
    
    Finally, the last statement follows from Proposition \ref{prop:221}.
\end{proof}

We notice that the potential-type spaces introduced above coincide with some Triebel-Lizorkin spaces of generalised smoothness for $1<p<\infty$, studied in \cite{Caetano-Moura}*{Section 2.2} and \cite{Moura}. Let us start by recalling their definition. Following the notation at the beginning of this section, let $\{\varphi_j\}_{j\geq 0}$ denote a resolution of the unity.
\begin{defi}%
Let $s\in\R$, $0<p<\infty$, $0<q\leq\infty$. Set $w$ for an admissible weight. We define the space $F_{p,q}^{s,w}(\R^n)$ to be the set of tempered distributions $f$ for which
\begin{equation*} %
    \norm{f}_{F_{p,q}^{s,w}(\R^n)}:=\norm{\left(\sum_{j=0}^\infty 2^{jsq}w(2^{-j})^q\abs{\varphi_j(D)f}^q\right)^{1/q}}_{L^p(\R^n)}<\infty.
\end{equation*}
\end{defi}

\begin{rem}%
    It was shown in \cite{Caetano-Moura} that the spaces $F_{p,q}^{s,w}(\R^n)$ are independent of the chosen resolution of the unity, in the sense that different resolutions, give rise to the same space with equivalent quasi-norms. 
\end{rem}

\begin{prop}\label{pot_as_triebel}
    Let $w$ be an admissible weight and $1<p<\infty$. Then the potential-type space $\J_{\Rw}(L^p(\R^n))$ coincides with the Triebel-Lizorkin space $F_{p,2}^{0,1/w}(\R^n)$ of generalised smoothness, with equivalent norms. Namely, it holds that
    \[
         \norm{f}_{\J_{\Rw}(L^p(\R^n))}\approx\norm{f}_{F_{p,2}^{0,1/w}(\R^n)},
    \]
    for all $f\in \J_{\Rw}(L^p(\R^n))={F_{p,2}^{0,1/w}(\R^n)}$.
\end{prop}

\begin{proof}
    The fact that the space $L^p(\R^n)$ and $F_ {p,2}^0(\R^n)$ coincide with equivalent norms (see e.g. \cite{Trie83}), and Proposition \ref{prop:inverse_w} yield
    \[
      \norm{f}_{\J_{\Rw}(L^p(\R^n))}\approx\norm{\J_{\mathrm{u}}f}_{F_ {p,2}^0(\R^n)},
    \]
    where $\mathrm{u}$ stands for the regularisation of $w^{-1}$. Finally, the lifting property  \cite{Caetano-Moura}*{Proposition 3.2} gives
    \[
      \norm{\J_{\mathrm{u}}f}_{F_ {p,2}^0(\R^n)}\approx\norm{f}_{F_{p,2}^{0,1/w}(\R^n)},
    \]
finishing the proof.
\end{proof}

To finish this section, we shall point out yet another connection of some of the potential-type spaces studied in this paper, with spaces existing in the literature. 

\begin{defi}\label{reg_logb} Let $b\in \R$. We shall denote by $w_b$ the admissible weight 
\[
    w_b(t)=\left(1+\log_+(1/t)\right)^b
\]
and by $\Rw_b$ its regularisation given in Definition \ref{Regularised_weight}, which following the notation there, can be explicitly written as
\[
    \Rw_b(\xi):=\sum_{j\geq 0} (1+j \log 2)^b\varphi_j(\xi).  
\]
\end{defi}

We would like to point out that the spaces $J_{\Rw_b}(L^2(\R^n))$ lay within the family of the so-called refined Sobolev scale $H^{\varphi}(\R^n)$, which also coincide with the H\"ormander spaces $B_{2,\mu}(\R^n)$ introduced by L. H\"ormander (see \cite{Horm_book}*{Definition 2.2.1} or \cite{Mikhailets}*{Definition 1.9}) for $\mu(\xi)=\varphi(\langle\xi\rangle)$, and appear in the study of elliptic operators. These spaces are defined as follows.

\begin{defi}\label{RefSob}
    \cite{Mikhailets}*{Definition 1.10} Let $\varphi:[1,\infty)\rightarrow (0,\infty)$ be a Borel measurable function for which both $\varphi$ and $1/\varphi$ are bounded in every compact interval of the form $[1,c]$, with $1<c<\infty$. In addition, let us assume that there exist $d\geq 1$ and a function  $\psi:[d,\infty)\rightarrow (0,\infty)$ satisfying the following two properties.
    \begin{enumerate}[label=(\roman*)]
        \item\label{eq:1} The function $\psi:[d,\infty)\rightarrow (0,\infty)$ is Borel measurable on $[d_0,\infty)$ for some $d_0\geq d$ and  for all $\lambda>0$
        \[
            \lim_{t\rightarrow\infty}\frac{\psi(\lambda t)}{\psi(t)}=1.
        \]
        \item It holds that $\varphi(t)\approx\psi(t)$ for all $t\geq d$.
    \end{enumerate}
    Then we define the space $H^{\varphi}(\R^n)$ as the set of tempered distributions $f$ for which their Fourier transform $\widehat{f}$ is locally integrable in $\R^n$ and satisfies
    \[
        \int_{\R^n}\abs{\varphi(\langle\xi\rangle)\widehat{f}(\xi)}^2\dd\xi<\infty,
    \]
    endowed with the norm
    \[
        \norm{f}_{H^{\varphi}(\R^n)}:=\left(\int_{\R^n}\abs{\varphi(\langle\xi\rangle)\widehat{f}(\xi)}^2\dd\xi\right)^{1/2}.
    \]
\end{defi}

More precisely, we have the following identification.
\begin{prop}\label{pot_as_hor}
    Let $b\in \R$. Then the potential-type space $\J_{\Rw_b}(L^2(\R^n))$ coincides with the refined Sobolev space $H^{\varphi}(\R^n)$, with $\varphi(t):=w_b(1/t)$, with equivalent norms. More precisely, it holds that
\[
    \norm{f}_{J_{\Rw_b}(L^2(\R^n))}\approx \norm{f}_{H^\varphi(\R^n)}
\]
for all $f\in\J_{\Rw_b}(L^2(\R^n))=H^\varphi(\R^n)$.
\end{prop}
\begin{proof} We shall prove first that the function $\varphi(t)=w_b(1/t)$ satisfies the conditions required in Definition \ref{RefSob}, with $d=1$ and  $\psi=\varphi$, which would yield \ref{eq:1}.

The measurably condition is ensured by the monotonicity of the weights.
    
Furthermore, we have that $w_b(1/t)$ and $w_b^{-1}(1/t)$ are both bounded on any interval of the form $[1,c]$ for all $c>1$. Indeed, if $b\geq 0$ then $w_b$ satisfies
    \[
        w_b(1/t)\leq w_b(1/c)\quad \mathrm{and}\quad \frac{1}{w_b(1/t)}\leq\frac{1}{w_b(1)}
    \]
    for all $1\leq t\leq c$, with $1<c<\infty$. A similar argument shows the property for $b<0$.
    
    To show \ref{eq:1}, we notice that, for all $\lambda>0$
    \[
        \lim_{t\rightarrow\infty}\frac{w_b(1/\lambda t)}{w_b(1/t)}=\left(\lim_{t\rightarrow\infty}\frac{1+\log_+(\lambda t)}{1+\log_+ t}\right)^b=\left(\lim_{t\rightarrow\infty}\frac{1+\log(\lambda t)}{1+\log t}\right)^b=1.
    \]
    
    To show the equivalence of norms, note that the Pancherel Theorem and   \eqref{sigma_equiv_weight} yield
    \[
        \norm{f}_{J_{\Rw_b}(L^2(\R^n))}^2\approx\int_{\R^n}\abs{w_b(1/\langle\xi\rangle)\widehat{f}(\xi)}^2\dd\xi=\norm{f}_{H^\varphi(\R^n)},
    \]
    finishing the proof.
\end{proof}
\begin{rem}
    Let $w$ be one of the weights considered in Remark \ref{remloglog} and denote by $\Rw$ its regularisation. Define $\varphi(t)=w(1/t)$ for $t\geq 1$. One can show that this function satisfies the hypothesis in Definition \ref{RefSob} since it behaves asymptotically like  $\psi(t)=\log^{b_1}(t)\log(\log(t))^{b_2}$ as $t>>1$, which satisfies the condition \ref{eq:1} in Definition \ref{RefSob} as shown in \cite{Mikhailets}*{Example 1.1}. Hence, arguing as in the previous result, one can show that the associated spaces $J_{\Rw}(L^2(\R^n))$ and $H^\varphi(\R^n)$ coincide.
\end{rem}

\section{Main results}\label{main_results}

It was shown in \cite{Paper1}*{Theorem 7.1}  that bilinear Coifman-Meyer multipliers map $X_w(\R^n)\times X_w(\R^n)$ continuously into the potential-type space of generalised smoothness $\J_{\Rw}(\BMO(\R^n))$. More specifically, one obtains the following from that result.

\begin{teo}
Let $\Symb(\xi,\eta)$ be a smooth function on $\R^n\times\R^n\setminus\lbrace(0,0
)\rbrace$ satisfying \eqref{CM_mult_property} and let $T_\Symb$ be the corresponding bilinear Coifman-Meyer multiplier. Let $w$ be an admissible weight satisfying $\inf_{t>0} w(t)>0$ and let $\Rw$ be its regularisation given in Definition \ref{Regularised_weight}. There exists a constant $C$ such that 
\[
    \norm{T_\Symb(f,g)}_{\J_{\Rw}(\BMO(\R^n))}\leq C\norm{f}_{X_w(\R^n)}\norm{g}_{X_w(\R^n)},
\]
for all $f,g\in X_w(\R^n)$. 
\end{teo}

The aim of the present article is to extend the boundedness range of these bilinear multipliers to the case when one of the two arguments of $T_\sigma$ belongs to the space $X_w(\R^n)$, while the other one is either in a Lebesgue space $L^p(\R^n)$, with $1<p<\infty$, or is an element in the Hardy space $H^1(\R^n)$. The obtained results involve, as in \cite{Paper 1}, potential-type spaces of generalised smoothness. On this occasion, Lebesgue and Hardy potential-type spaces arise. The following is the main result of this paper.

\begin{teo}\label{CM_main_thm}
Let $\Symb(\xi,\eta)$ be a smooth function on $\R^n\times\R^n\setminus\lbrace(0,0
)\rbrace$ satisfying \eqref{CM_mult_property} and let $T_\Symb$ be the corresponding bilinear Coifman-Meyer multiplier. Let $w$ be an admissible weight satisfying $\inf_{t>0} w(t)>0$ and let $\Rw$ be its regularisation given in Definition \ref{Regularised_weight}.
\begin{enumerate}[label=(\roman*)]
    \item\label{itm:ecu22} Given $1<p<\infty$, there exists a constant $C>0$ such that 
    \[
        \norm{T_\Symb(f,g)}_{\J_{\Rw}(L^p(\R^n))}\leq C\norm{f}_{L^p(\R^n)}\norm{g}_{X_w(\R^n)}
    \]
    holds for every $f\in L^p(\R^n)$ and $g\in X_w(\R^n)$.
    \item\label{itm:ecu23} The symbol $\Symb$ can be decomposed as the sum of two symbols $\Symb=\Symb_g+\Symb_b$, such that we can find constants $C',C''>0$ for which
    \[
        \norm{T_{\Symb_g}(f,g)}_{L^1(\R^n)}\leq C'\norm{f}_{H^1(\R^n)}\norm{g}_{X_w(\R^n)}
    \]
    and
    \[
        \norm{T_{\Symb_b}(f,g)}_{\J_{\Rw}(H^1(\R^n))}\leq C''\norm{f}_{H^1(\R^n)}\norm{g}_{X_w(\R^n)}
    \]
    hold for every $f\in H^1(\R^n)$ and $g\in X_w(\R^n)$.
\end{enumerate}
\end{teo}

\begin{rem}
    Not that if $w\equiv 1$, Proposition \ref{prop:gather} implies that $X_w(\R^n)=L^\infty(\R^n)$, and it is easily shown that $\J_{\Rw}(H^1(\R^n))=H^1(\R^n)\subset L^1(\R^n)$. Hence Theorem \ref{CM_main_thm}  
    recovers the known boundeness  results $L^p(\R^n)\times L^\infty(\R^n)\rightarrow L^p(\R^n)$, and $H^1(\R^n)\times L^\infty(\R^n)\rightarrow L^1(\R^n)$ (see e.g. \cites{CMbook,Grafakos-Torres}).
\end{rem}

\begin{rem}\label{min_der}
Analysing the proof of the result above, we note that the conclusions of the theorem, can be achieved by requiring $\sigma$ to satisfy \eqref{CM_mult_property} only for multi-indices $\alpha,\beta$ such that $|\alpha|+|\beta|\leq 4n+1$.
\end{rem}

If we consider the logarithmic admissible weight $w_1$ from Example \ref{exe-logb}, and taking into consideration that by Proposition \ref{prop:gather} we have that $\bmo(\R^n)=X_{w_1}(\R^n)$, then Theorem \ref{CM_main_thm} yields the following endpoint estimates.

\begin{cor}\label{Cor:main_bmo}
    Let $w_1(t)=1+\log_+1/t$ and let $\Rw_1$ be its regularisation from Definition \ref{reg_logb}. Let $\Symb$ and $T_\Symb$ be as in Theorem \ref{CM_main_thm}. 
    
    \begin{enumerate}
        \item Given $1<p<\infty$ we can find a constant $C>0$ such that the estimate
    \[
        \norm{T_\Symb(f,g)}_{\J_{\Rw_1}(L^p(\R^n))}\leq C\norm{f}_{L^p(\R^n)}\norm{g}_{\bmo(\R^n)}
    \]
    holds for every $f\in L^p(\R^n)$ and $g\in\bmo(\R^n)$. 
    
    \item The symbol $\Symb$ can be decomposed as the sum of two symbols $\Symb=\Symb_g+\Symb_b$ such that we can find constants $C',C''>0$ for which the estimates
    \[
        \norm{T_{\Symb_g}(f,g)}_{L^1(\R^n)}\leq C'\norm{f}_{H^1(\R^n)}\norm{g}_{\bmo(\R^n)}
    \]
    and
    \[
        \norm{T_{\Symb_b}(f,g)}_{\J_{\Rw_1}(H^1(\R^n))}\leq C''\norm{f}_{H^1(\R^n)}\norm{g}_{\bmo(\R^n)}
    \]
    hold for every $f\in H^1(\R^n)$ and $g\in \bmo(\R^n)$.
    \end{enumerate}
\end{cor}

\section{Boundedness of paraproducts}\label{SecPara}

The strategy we follow to prove our main results, relies on obtaining estimates for paraproducts of the type
\[
    \Pi(f,g)(x):=\int_0^\infty (Q_tf)(x)(P_tg)(x)m(t)\frac{\dd t}{t},
\]
where $m(t)$ is a measurable bounded function on $(0,\infty)$ and $Q_t$, $P_t$ are frequency localisation operators, defined as follows. Given $\psi\in \SW(\R^n)$, whose Fourier transform is supported in a ring satisfying 
\begin{equation*}
    \int_0^\infty\abs{\widehat{\psi}(t\xi)}^2\frac{\dd t}{t}<\infty,\quad \mbox{for all $\xi\neq 0$},
\end{equation*}
and given $\phi\in \SW(\R^n)$, whose Fourier transform is supported in a ball centred at the origin, we define the frequency localisation operators
\begin{align*}
Q_tf:=\widehat{\psi}(tD)f,\qquad P_tf:=\widehat{\phi}(tD)f,
\end{align*}
where $t>0$ and $f\in\SW(\R^n)$. 

To study these paraproducts, it is convenient to decompose them one step further. To this end, let $r<R$ be two positive real numbers such that $\widehat{\psi}$ is supported in the ring $\lbrace r\leq\abs{\xi}\leq R\rbrace$ and $\widehat{\phi}$ is supported in the ball $\lbrace\abs{\xi}\leq R\rbrace$. We write $\phi$ as $\phi=\psi^{(1)}+\phi^{(1)}$, with $\widehat{\psi^{(1)}}$ being supported in the annulus $\lbrace 2r/3\leq\abs{\xi}\leq R\rbrace$ and $\widehat{\phi^{(1)}}$ supported in the ball $\lbrace\abs{\xi}\leq r/2\rbrace$. In addition, we pick a Schwartz function $\phi^{(2)}$ whose Fourier transform is supported in a ball and is identically one on $\lbrace\abs{\xi}\leq 2R\rbrace$, while $\psi^{(2)}$ will denote a radial Schwartz function whose Fourier transform is supported in an annulus and is identically one on $\lbrace r/2\leq\abs{\xi}\leq 3R/2\rbrace$.

We can then write $\Pi$ as the sum of two bilinear operators,
\begin{equation}\label{eq:decomp}
    \Pi(f,g)=\Pi_1(f,g)+\Pi_2(f,g)
\end{equation}
where
\begin{equation}\label{eq:Pi1}
    \Pi_1(f,g)(x)=\int_0^\infty Q_t^{(2)}[ (Q_tf)(P_t^{(1)}g)](x)m(t)\frac{\dd t}{t}
\end{equation}
and
\begin{equation}\label{eq:Pi2}
    \Pi_2(f,g)(x)=\int_0^\infty P_t^{(2)}[ (Q_tf)(Q_t^{(1)}g)]m(t)\frac{\dd t}{t}.
\end{equation}
Here $Q_t^{(i)}$ and $P_t^{(i)}$ denote the frequency localisation operators associated to $\psi^{(i)}$ and $\phi^{(i)}$ respectively, with $i=1,2$.

The boundedness properties of these operators on $X_w(\R^n)\times X_w(\R^n)$ was studied in \cite{Paper1}*{Theorem 5.2}. More specifically the following result is consequence of the proof of that theorem:
\begin{prop}\label{endpoint_results_BMO}
Let $w$ be an admissible  weight satisfying $\inf_{t>0} w(t)>0$ and let $\Rw$ be its regularisation given in Definition \ref{Regularised_weight}. We can find constants $C',C''>0$ such that 
    \begin{equation*}%
    \left\Vert\Pi_1(f,g)\right\Vert_{\J_{\Rw}(\BMO(\R^n))}\leq C'\norm{f}_{\BMO(\R^n)}\norm{g}_{X_w(\R^n)}
    \end{equation*}
    holds for all $f\in \BMO(\R^n)$ and $g\in X_w(\R^n)$, while
    \begin{equation*}%
    \left\Vert\Pi_2(f,g)\right\Vert_{\BMO(\R^n)}\leq C''\norm{f}_{\BMO(\R^n)}\norm{g}_{\BMO(\R^n)}
    \end{equation*}
    holds for all $f,g\in \BMO(\R^n)$.
    In consequence, as $X_w(\R^n)\subset \BMO(\R^n)$, 
    \[
       \left\Vert\Pi(f,g)\right\Vert_{J_{\Rw}(\BMO(\R^n))}\lesssim \norm{f}_{\BMO(\R^n)}\norm{g}_{X_w(\R^n)}, 
    \]
holds for all $f\in \BMO(\R^n)$ and $g\in X_w(\R^n)$.
\end{prop}

We are interested in finding estimates for the paraproduct $\Pi$ when one of the arguments belongs to the space $X_w(\R^n)$, while the other function belongs to either a Lebesgue space $L^p(\R^n)$, $1<p<\infty$, or the Hardy space $H^1(\R^n)$. 

When the term in $X_w(\R^n)$ lies in the first argument, these type of boundedness properties are a direct consequence of results already existing in the literature, and we summarise them in the following lemma. 

\begin{lem}\label{lem:faciles} Let $w$ be an admissible  weight satisfying $\inf_{t>0} w(t)>0$ and let $\Rw$ be its regularisation given in Definition \ref{Regularised_weight}.
\begin{enumerate}[label=(\roman*)]
    \item \label{Para_BMO_Lp}
    Let $1<p<\infty$. There is a constant $C>0$ such that
    \[
        \norm{\Pi(f,g)}_{L^p(\R^n)}\leq C\norm{f}_{X_w(\R^n)}\norm{g}_{L^p(\R^n)}
    \]
    for all $f\in X_w(\R^n)$ and $g\in L^p(\R^n)$.
    \item \label{Para_H1_Lp}There is a constant $C>0$ such that
    \[
        \norm{\Pi(f,g)}_{L^1(\R^n)}\leq C\norm{f}_{X_w(\R^n)}\norm{g}_{H^1(\R^n)}
    \]
    for all $f\in X_w(\R^n)$ and $g\in H^1(\R^n)$.
\end{enumerate}
\end{lem}
\begin{proof}
If a function $f\in X_w(\R^n)\subset \BMO(\R^n)$ then the linear operator $g\mapsto\Pi(f,g)$ is of Calder\'on-Zygmund-type (see e.g. \cite{Grafakos2}*{Section 4}), and hence it is bounded on any $L^p(\R^n)$ for $1<p<\infty$, and from $H^1(\R^n)$ to $L^1(\R^n)$ with operator norm at most a multiple of $\norm{f}_{\BMO(\R^n)}$, which is smaller or equal than $\norm{f}_{X_w(\R^n)}$.
\end{proof}

To complete the picture, we need to study the case where the second argument belongs to $X_w(\R^n)$, while the first one is considered to be in either a Lebesgue space $L^p(\R^n)$, with $1<p<\infty$, or the Hardy space $H^1(\R^n)$.

\begin{teo}\label{bmo_Lp}
   Let $w$ be an admissible  weight satisfying $\inf_{t>0} w(t)>0$ and let $\Rw$ be its regularisation given in Definition \ref{Regularised_weight}.
    \begin{enumerate}[label=(\roman*)]
    
    \item\label{itm:ecu16} Let $1<p<\infty$. There is a constant $C>0$ such that 
    \[
        \norm{\Pi(f,g)}_{\J_{\Rw}(L^p(\R^n))}\leq C\norm{f}_{L^p(\R^n)}\norm{g}_{X_w(\R^n)}
    \]
    holds for every $f\in L^p(\R^n)$ and $g\in X_w(\R^n)$.
    
    \item\label{itm:ecu_H1_bmo} 
    We can find constants $C',C''>0$ such that
     \begin{equation*}%
    \left\Vert\Pi_1(f,g)\right\Vert_{\J_{\Rw}(H^1(\R^n))}\leq C''\norm{f}_{H^1(\R^n)}\norm{g}_{X_w(\R^n)}
    \end{equation*}
    and
    \begin{equation*}%
    \left\Vert\Pi_2(f,g)\right\Vert_{L^1(\R^n)}\leq C'\norm{f}_{H^1(\R^n)}\norm{g}_{\BMO(\R^n)}
    \end{equation*}
    hold for every $f\in H^1(\R^n)$ and $g\in X_w(\R^n)$.
    
    \end{enumerate}
\end{teo}

\begin{proof}[{Proof of Theorem \ref{bmo_Lp}}]

Let us begin by proving the last part of the second statement. To this end lets fix first $f\in H^1(\R^n)$ and $g\in  \BMO(\R^n)$. 

We shall prove that $\Pi_2(f,g)\in L^1(\R^n)$. By duality, it is enough to show that, for $H\in L^\infty(\R^n)$, we can estimate the expression
\[
    \langle\Pi_2(f,g),H\rangle=\int\!\!\!\!\int_0^\infty (Q_tf)(x)(Q_t^{(1)}g)(x)(P^{(2)}_tH)(x)m(t)\frac{\dd t}{t}\dd x.
\]
Note that Theorem \ref{rem_BMO_CM} yields that $|{(Q_t^{(1)}g)(x)}|^2 t^{-1}\dd t\dd x$ is a Carleson measure, whose norm is bounded by a constant times $\norm{g}_{\BMO}^2$. Moreover, we have that
\[
    \sup_{t>0,x\in\R^n}\abs{(P^{(2)}_tH)(x)m(t)}\leq \norm{m}_\infty \norm{H}_{L^\infty(\R^n)}.
\]
Therefore, Proposition \ref{hardy_CM} yields that
\[
    \abs{\langle\Pi_2(f,g),H\rangle}\lesssim\norm{m}_{L^\infty(\R^n)}\norm{f}_{H^1(\R^n)}\norm{g}_{\BMO(\R^n)}\norm{H}_{L^\infty(\R^n)}.
\]
It follows that $\Pi_2(f,g)$ belongs to $L^1(\R^n)$ and
\begin{equation*}
    \norm{\Pi_2(f,g)}_{L^1(\R^n)}\lesssim\norm{m}_{L^\infty(\R^n)}\norm{f}_{H^1(\R^n)}\norm{g}_{\BMO(\R^n)}.
\end{equation*}
To study the stated boundedness for $\Pi_1$, let us fix $f\in H^1(\R^n)$ and $g\in  X_w(\R^n)$. Consider now a function $H$ of the form $H=J_{\Rw^{-1}}h$ with $h\in\BMO(\R^n)$. By duality it is enough to estimate the expression
\[
    \langle\Pi_1(f,g),H\rangle=\int\!\!\!\!\int_0^\infty (Q_tf)(x)w(t)(Q^{(2)}_tH)(x)v(t,x)\frac{\dd t}{t}\dd x,
\]
where $v(t,x)=m(t)(P_t^{(1)}g)(x)/w(t)$. By the definition of the norm on $X_w(\R^n)$, we have that 
\begin{equation}\label{eq:pepe}
        \abs{v(t,x)}\lesssim\norm{m}_{L^\infty(\R^n)}\norm{g}_{X_w(\R^n)}, \quad \mbox{for all $t>0$ and  $x\in\R^n$}.
\end{equation}

 Next we can write $w(t)(Q^{(2)}_tH)(x)=(R_th)(x)$, where $R_t$ is the integral operator defined by
\[
(R_tF)(x):=\int_{\R^n} K_t(x,y)F(y)\dd y
\]
with kernel $K_t(x,y):=J_t(x-y)$ and
\[
    J_t(z)=w(t)\int_{\R^n}\frac{\widehat{\psi^{(2)}}(t\xi)}{\Rw(\xi)}e^{iz\xi}\ddd\xi.
\]
Let us now show that the linear operators $R_t$ and their kernels $K_t$ satisfy the hypotheses of Theorem \ref{IO_BMO_CM}. This would imply that the measure defined by 
\begin{equation}\label{eq:gotera}
        \abs{w(t)(Q^{(2)}_tH)(x)}^2\dd x\frac{\dd t}{t},
\end{equation}
is a Carleson measure with norm bounded by a constant times $\norm{h}_{\BMO(\R^n)}^2$.

To this end, we begin by observing that $R_t1\equiv 0$ for all $t>0$. To find the kernel estimates, a change of variables and integration by parts yield
\begin{align}\label{ecu27}
    J_t(z)&=\frac{w(t)}{t^n}\int_{\R^n}\frac{\widehat{\psi^{(2)}}(\xi)}{\Rw(\xi/t)}e^{iz\xi/t}\ddd\xi \nonumber\\
    &=\frac{w(t)}{t^n}\left(\frac{\abs{z}}{t}\right)^N\int_{\R^n}(-\Delta)^N\left[\frac{\widehat{\psi^{(2)}}(\xi)}{\Rw(\xi/t)}\right]e^{iz\xi/t}\ddd\xi
\end{align}
for any $N\geq 1$. Note that the Leibniz rule, the fact that $\abs{\xi}\approx 1$ and \eqref{ecu2sym} give
\begin{equation}\label{ecu26}
    \abs{(-\Delta)^N\left[\frac{\widehat{\psi^{(2)}}(\xi)}{\Rw(\xi/t)}\right]}\lesssim\frac{1}{w(t)}.
\end{equation}
Finally, using \eqref{ecu27}, \eqref{ecu26} and the fact that $\widehat{\psi^{(2)}}$ is compactly supported we obtain for any integer $N>n/2$ that
\begin{equation*}%
    \abs{J_t(z)}\lesssim t^{-n}\frac{t^{2N}}{(t+\abs{z})^{2N}},
\end{equation*}
from where the estimates for $K_t$ follows.

Finally, the Plancherel theorem  and \eqref{sigma_equiv_weight} yield the quadratic estimate
\begin{align*}
    \int_0^\infty\!\!\!\!\int\abs{(R_tf)(x)}^2\dd x\frac{\dd t}{t}&\approx\int_0^\infty\!\!\!\!\int\abs{Q_t^{(2)}f(x)}^2\dd x\frac{\dd t}{t}\lesssim\norm{f}_{L^2(\R^n)}^2.
\end{align*}
Then \eqref{eq:pepe}, \eqref{eq:gotera} and Proposition \ref{hardy_CM} imply
\begin{align*}
    \abs{\langle\Pi_1(f,g),J_{\Rw^{-1}}h\rangle}&=\abs{\int\!\!\!\!\int_0^\infty (Q_tf)(x)w(t)(Q^{(2)}_tH)(x)v(t,x)\frac{\dd t}{t}\dd x}\nonumber\\
    &\lesssim\norm{m}_{L^\infty(\R^n)}\norm{f}_{H^1(\R^n)}\norm{g}_{X_w(\R^n)}\norm{h}_{\BMO(\R^n)}.
\end{align*}
It follows by duality that $\J_{\Rw^{-1}} \Pi_1(f,g)$ belongs to the Hardy space $H^1(\R^n)$ and
\[
    \norm{\Pi_1(f,g)}_{\J_{\Rw}H^1(\R^n)}=\norm{\J_{\Rw^{-1}} \Pi_1(f,g)}_{H^1(\R^n)}\lesssim\norm{m}_{L^\infty(\R^n)}\norm{f}_{H^1(\R^n)}\norm{g}_{X_w(\R^n)},
\]
which shows the claimed estimate for $\Pi_1$.

Finally, let us show the first part of the statement. To this aim, notice that for a fixed $g\in X_w(\R^n)$, if we consider the linear operator $f\mapsto \J_{\Rw^{-1}}\Pi_1(f,g)$, by complex interpolation  between the results above (see. e.g. \cite{Janson_Jones}), and those in Proposition \ref{endpoint_results_BMO}, it follows that for all $1<p<\infty$ and for all $f\in L^p(\R^n)$ holds that
\[
    \left\Vert\Pi_1(f,g)\right\Vert_{\J_{\Rw}(L^p(\R^n))}\leq C''\norm{f}_{L^p(\R^n)}\norm{g}_{X_w(\R^n)},
\]
with constant independent on $f$ or $g$. 
Similarly, one obtains that for all $1<p<\infty$
\[
    \left\Vert\Pi_2(f,g)\right\Vert_{L^p(\R^n)}\leq C''\norm{f}_{L^p(\R^n)}\norm{g}_{\BMO(\R^n)}
\]
holds for all $f\in L^p(\R^n)$ and $g\in \BMO$.

We know from Proposition \ref{prop:221} that $L^p(\R^n)\subset \J_{\Rw}(L^p(\R^n))$, from where it follows that $L^p(\R^n)+\J_{\Rw}(L^p(\R^n))=\J_{\Rw}(L^p(\R^n))$. Using this, and the estimates above one shows the first statement of the theorem.  
\end{proof}

\section{Proof of Theorem \ref{CM_main_thm}}\label{CFM_section}

Proceeding as in the proof of  \cite{CMbook}*{Proposition 2}, let us start by decomposing the symbol $\Symb$ as the sum of two symbols,
\[
    \Symb(\xi,\eta)=\tau_1(\xi,\eta)+\tau_2(\xi,\eta),
\]
where $\tau_1$ and $\tau_2$ still satisfy \eqref{CM_mult_property} and $\tau_1$ is supported in $\lbrace\abs{\xi}\geq\abs{\eta}/20\rbrace$, while $\tau_2$ is supported in $\lbrace\abs{\xi}\leq\abs{\eta}/10\rbrace$, 

Next we consider a Schwartz function $\psi$ which is frequency supported in the ring $\lbrace 4/5\leq\vert\xi\vert\leq 6/5\rbrace$ and satisfies
\begin{equation*}
    \int_0^\infty\vert\widehat{\psi}(t\xi)\vert^2\frac{\dd t}{t}=1
\end{equation*}
for all $\xi\neq 0$. In addition, we select a Schwartz function $\phi$ whose Fourier transform is supported in a ball, and is identically one in the frequency support of ${\psi}$. Then $\tau_1$ can be written as
\begin{align*}
    T_{\tau_1}(f,g)(x)=\int_0^\infty\!\!\!\!\int_{\R^{2n}}\tilde{\tau_1}(t\xi,t\eta)\widehat{Q_tf}(\xi)\widehat{P_tg}(\eta)e^{i(x\xi+x\eta)}\ddd\xi\ddd\eta\frac{\dd t}{t},
\end{align*}
where $\tilde{\tau_1}(\xi,\eta)=\tau_1(\xi/t,\eta/t)\widehat{\psi}(\xi)\widehat{\phi}(\eta).$

For a fixed $t>0$, the function $\tilde{\tau_1}$ is smooth and compactly supported away from the origin. Hence, for $N\in\N$, the Fourier inversion formula yields
\begin{equation}\label{ecu1.1}
    \tilde{\tau_1}(t\xi,t\eta)=\int_{\R^{2n}} m_1(t,u,v)e^{i(u\xi+v\eta)}\frac{\dd (u,v)}{(1+\vert u\vert^2+\vert v\vert^2)^N},
\end{equation}
where
\[
    m_1(t,u,v)=t^{-2n}\widehat{\tilde{\tau_1}}(u/t,v/t)(1+\vert u\vert^2+\vert v\vert^2)^N.
\]
By \eqref{ecu1.1}, the operator $T_{\tau_1}$ can be expressed as
\begin{align*}
    &T_{\tau_1}(f,g)(x)\\
    &=\int_{\R^{2n}}\!\int_0^\infty\!\!\!\!\int_{\R^{2n}} m_1(t,u,v)\widehat{Q_tf}(\xi)\widehat{P_tg}(\eta)e^{i(\xi(u+x)+\eta(v+x))}\dd(\xi,\eta)\frac{\dd t}{t}\frac{\dd(u,v)}{(1+\vert u\vert^2+\vert v\vert^2)^N}\\
    &=\int_{\R^{2n}}\!\int_0^\infty\!\!\!\!\int_{\R^{2n}} m_1(t,u,v)\widehat{Q^{u}_tf}(\xi)\widehat{P^{v}_tg}(\eta)e^{i(\xi x+\eta x)}\dd(\xi,\eta)\frac{\dd t}{t}\frac{\dd (u,v)}{(1+\vert u\vert^2+\vert v\vert^2)^N}\\
    &=\int_{\R^{2n}}\!\int_0^\infty (Q^{u}_tf)(P^{v}_tg)m_1(t,u,v)\frac{\dd t}{t}\frac{\dd(u,v)}{(1+\vert u\vert^2+\vert v\vert^2)^N},
\end{align*}
where $Q^{u}_t$ and $P^v_t$ are the frequency localisation operator associated to $\psi^u(x):=\psi(x+u)$ and  $\phi^{v}(x):=\phi(x+v)$ respectively. An integration by parts argument, jointly with  \eqref{CM_mult_property} for $\tau_1$, shows that $m_1(t,u,v)$ is uniformly bounded in its three variables. Also we can show that for all $\delta>0$, 
\[
	\abs{\psi^u(x)}\lesssim \frac{(1+\abs{u})^{n+\delta}}{(1+\abs{x})^{n+\delta}}\quad {\rm and} \quad \abs{\phi^v(x)}\lesssim \frac{(1+\abs{v})^{n+\delta}}{(1+\abs{x})^{n+\delta}}.
\]
Next we observe that the bilinear operator
\begin{equation*}
    \Pi^{u,v}(f,g)(x)=\int_0^\infty (Q^{u}_tf)(P^v_tg)m_1(t,u,v)\frac{\dd t}{t}
\end{equation*}
is similar to those studied in Section \ref{SecPara}. Proceeding as in \eqref{eq:decomp}, we can write 
\[
    \Pi^{u,v}(f,g)=\Pi^{u,v}_{1}(f,g)+\Pi^{u,v}_{2}(f,g).
\]

Let us show the validity of part \ref{itm:ecu23} in the statement. To this end, fix two functions $f\in H^1(\R^n)$ and $g\in X_w(\R^n)$. Theorem~\ref{bmo_Lp}\ref{itm:ecu_H1_bmo} shows that
\begin{equation}\label{ecu25}
\norm{\Pi^{u,v}_{1}(f,g)}_{\J_{\Rw} H^1(\R^n)}\lesssim \mathcal{P}^{(1)}(u,v)\norm{f}_{H^1(\R^n)}\norm{g}_{X_w(\R^n)},
\end{equation}
and $\Pi^{u,v}_{2}$ satisfies
\begin{equation}\label{ecu24}
\norm{\Pi^{u,v}_{2}(f,g)}_{L^1(\R^n)}\lesssim \mathcal{P}^{(2)}(u,v)\norm{f}_{H^1(\R^n)}\norm{g}_{X_w(\R^n)}.
\end{equation}
Here $\mathcal{P}^{(1)}(u,v)$ and $\mathcal{P}^{(2)}(u,v)$ are polynomials in $\abs{u}$ and $\abs{v}$, independent of the functions $f$ and $g$.

Consequently, the operator $T_{\tau_1}$ can be written as $\tau_1=\tau_1^{(1)}+\tau_1^{(2)}$, where
\[
    T_{\tau_1^{(i)}}(f,g)(x)=\int_{\R^{2n}}\Pi^{u,v}_{i}(f,g)(x)\frac{\dd(u,v)}{(1+\vert u\vert^2+\vert v\vert^2)^N}, \quad i=1,2.
\]
In particular, by choosing $N$ large enough, the Minkowskii integral inequality, \eqref{ecu25} and \eqref{ecu24} yield 
\[
    \norm{T_{\tau_1^{(1)}}(f,g)}_{\J_{\Rw}(H^1(\R^n))}\lesssim\norm{f}_{H^1(\R^n)}\norm{g}_{X_w(\R^n)}
\]
and
\[
    \norm{T_{\tau_1^{(2)}}(f,g)}_{L^1(\R^n)}\lesssim\norm{f}_{H^1(\R^n)}\norm{g}_{X_w(\R^n)}.
\]

Proceeding in a similar same way as we did for $T_{\tau_1}$, but interchanging the roles of $\xi$ and $\eta$, we have
that we can write the operator  $T_{\tau_2}$ as 
\[
    T_{\tau_2}(f,g)(x)=\int_{\R^{2n}}\int_0^\infty(P^u_tf)(Q^v_tg)m_2(t,u,v)\frac{\dd t}{t}\frac{\dd(u,v)}{(1+\vert u\vert^2+\vert v\vert^2)^N}.
\]
Lemma  \ref{lem:faciles}\ref{Para_H1_Lp} yields that the paraproduct
\[
    \Pi^{u,v}(g,f)(x)=\int_0^\infty(P^u_tf)(Q^v_tg)m_2(t,u,v)\frac{\dd t}{t}
\]
satisfies the estimate 
\[
    \norm{\Pi^{u,v}(g,f)}_{L^1}\lesssim \mathcal{P}(u,v) \norm{f}_{H^1(\R^n)}\norm{g}_{X_w(\R^n)},
\]
where $\mathcal{P}(u,v)$ is a polynomial in $\abs{u}$ and $\abs{v}$, independent of $f$ and $g$. By choosing $N$ large enough, and using the Minkowskii integral inequality, we conclude that 
\[
    \norm{T_{\tau_2}(f,g)}_{L^1(\R^n)}\lesssim\norm{f}_{H^1(\R^n)}\norm{g}_{X_w(\R^n)}.
\]
The proof of part \ref{itm:ecu23} in the statement finishes by taking $\Symb_g=\tau_1^{(2)}+\tau_2$ and $\Symb_b=\tau_1^{(1)}$.

Part \ref{itm:ecu22} of the theorem is proved analogously by combining Lemma \ref{lem:faciles}\ref{Para_BMO_Lp} and Theorem \ref{bmo_Lp}\ref{itm:ecu16}.

\section{Applications}\label{Applications}

In this section we will derive some consequences  from Corollary \ref{Cor:main_bmo}. Namely, we will give some endpoint inequalities of Kato-Ponce-type missing in the literature, and the reconstruction of the product of functions in the local space version of $\bmo(\R^n)$ and the local Hardy space $h^1(\R^n)$.

Throughout this section we shall fix the admissible weight
\[
    w_1(t)=1+\log_+1/t
\]
and $\Rw_1$ denotes its regularisation from Definition \ref{reg_logb}.

\subsection{Kato-Ponce inequalities}

For a given real number $s$ we define the fractional Laplacian operator $\J^{s}=(1-\Delta)^{s/2}$ as
\[
    (\widehat{\J^{s}f})(\xi):=(1+\abs{\xi}^2)^{s/2}\widehat{f}(\xi),\quad \xi\in\R^n,\quad f\in\SW(\R^n).
\]
Some Kato-Ponce-type estimates (see \cite{Koezuka_Tomita}*{Corollary 1.2} or  \cite{Naibo-Thomson}*{Corollary 2.6}) are known for local Hardy spaces. In particular, for $1\leq p<\infty$, the endpoint estimate
\begin{equation*} %
    \norm{J^{s}(fg)}_{h^p(\R^n)}\lesssim\norm{J^{s}f}_{h^{p}(\R^n)}\norm{g}_{L^{\infty}(\R^n)}+\norm{f}_{h^{p}(\R^n)}\norm{J^{s}g}_{L^{\infty}(\R^n)}
\end{equation*}
is known to hold for all $f,g\in\SW(\R^n)$, provided $s>0$.

Corollary \ref{Cor:main_bmo} allows us to extend this estimate to the case where one replaces the $L^\infty(\R^n)$ norm, by $\bmo(\R^n)$, provided that $s$ is large enough. More precisely, we obtain the following result.
\begin{cor}\label{Kato-Ponce}
Let $s>4n+1$. For all  $1<p<\infty$, 
\[
    \norm{J^{s}(fg)}_{\J_{\Rw_1}(L^p(\R^n))}\lesssim \norm{J^{s}f}_{L^p(\R^n)}\norm{g}_{\bmo(\R^n)}+\norm{f}_{L^p(\R^n)}\norm{J^{s}g}_{\bmo(\R^n)}
\]
holds for every $f,g\in\SW(\R^n)$. 
Moreover, it holds that
\[
   \norm{J^{s}(fg)}_{L^1(\R^n)+\J_{\Rw_1}(H^1(\R^n))}\lesssim\norm{J^{s}f}_{H^1(\R^n)}\norm{g}_{\bmo(\R^n)}+\norm{f}_{H^1(\R^n)}\norm{J^{s}g}_{\bmo(\R^n)}
\]
for every $f,g\in\SW(\R^n)$. 
\end{cor}

\begin{proof}
Following the argument in \cite{Grafakos-Oh}*{Theorem 1}, one decomposes $J^s(fg)$ as 
\[
    J^s(fg):=B^s_1(J^sf,g)(x)+B^s_2(f,J^sg)(x)+B^s_3(f,J^sg)(x),
\]
where $B^s_j$ are bilinear Fourier multipliers. It was also shown there, that the symbol of both $B^s_1$ and $B^s_2$ satisfy  \eqref{CM_mult_property} for all multi-indices $\alpha,\beta$.

Moreover,  following the argument in \cite{Gra-Mal_Nai}* {Section 4}, one shows that $B_3^s$ satisfies \eqref{CM_mult_property} for $|\alpha|+|\beta|\leq s$. Then, in order to apply the results obtained in this paper, as it was pointed out in Remark \ref{min_der}, we require $s>4n+1$.

Then, the result is a direct application of  Corollary \ref{Cor:main_bmo} to each term.

\end{proof}

\subsection{Product of functions}

The Corollary \ref{Cor:main_bmo} above enables us to study the product of a function in $\bmo(\R^n)$ with a function in either the Lebesgue space $L^p(\R^n)$, with $1<p<\infty$, the Hardy space $H^1(\R^n)$ or the local Hardy space $h^1(\R^n)$.

Indeed, taking the symbol $\sigma\equiv 1$, the following result, that corresponds to Corollary \ref{Kato-Ponce} for $s=0$, follows directly from Corollary \ref{Cor:main_bmo}.

\begin{cor}\label{cor:Lp_potential_prod} Let $1<p<\infty$. Then, for all $f\in L^p(\R^n)$ and $g\in\bmo(\R^n)$ it holds that
    \[
        \norm{fg}_{\J_{\Rw_1}(L^p(\R^n))}\lesssim\norm{f}_{L^p(\R^n)}\norm{g}_{\bmo(\R^n)}.
    \]
\end{cor}

The case where one of the terms belongs to $h^1(\R^n)$ and the other one to $\bmo(\R^n)$, is not a direct consequence of our results. Nevertheless, the relation between $H^1(\R^n)$  and its local version, as well as some properties of the latter space and local $\bmo(\R^n)$ allows to obtain the following.  

\begin{cor}\label{teo_prod} 
    There exist two continuous bilinear operators on the product space $h^1(\R^n)\times \bmo(\R^n)$,  respectively $B_1:h^1(\R^n)\times \bmo(\R^n)\to L^1(\R^n)$ and $B_2:h^1(\R^n)\times \bmo(\R^n)\to J_{\Rw_1}( H^1(\R^n))$, such that 
    \[
        fg=B_1(f,g)+B_2(f,g).
    \]
\end{cor}

\begin{proof}
 Pick $\Phi$ any Schwartz function for which $\int\Phi=1$. Let us consider a function $f$ in $h^1(\R^n)$ and a function $g$ in $\bmo(\R^n)$. So we can decompose $f$ as
\[
    f=(\Phi\ast f)+(f-\Phi\ast f)=:\mathcal{L}f+\mathcal{H}f.
\]
So, at least formally, we have that we could define the product of $f$ and $g$ as
\begin{equation}\label{def_prod}
    fg:=(\mathcal{L}f) g+ (\mathcal{H}f) g,
\end{equation}
provided we could make sense of the right hand side term of this expression. 

Notice first that  \cite{Goldberg}*{Lemma 4} implies that
\begin{equation}\label{boundedness}
     \norm{\mathcal{H}f}_{H^1(\R^n)}
    \lesssim\norm{f}_{h^1(\R^n)}.
\end{equation}
We can interpret the second term in \eqref{def_prod} as a bilinear Coifman-Meyer multiplier, where the symbol is identically one, acting on the product space $H^1(\R^n)\times \bmo(\R^n)$. In this way, Corollary \ref{Cor:main_bmo} provides a decomposition
\[
    (\mathcal{H} f)g=B_1^{(1)}(f,g)+B_2(f,g)
\]
where $B_1^{(1)}$ and $B_2$ are two bilinear operators satisfying the estimates
\[
    \norm{B_1^{(1)}(f,g)}_{L^1(\R^n)}\lesssim\norm{\mathcal{H} f}_{H^1(\R^n)}\norm{g}_{\bmo(\R^n)}\lesssim\norm{f}_{h^1(\R^n)}\norm{g}_{\bmo(\R^n)}
\]
and
\[
    \norm{B_2(f,g)}_{\J_{\Rw_1}(H^1(\R^n))}\lesssim\norm{\mathcal{H}f}_{H^1(\R^n)}\norm{g}_{\bmo(\R^n)}\lesssim\norm{f}_{h^1(\R^n)}\norm{g}_{\bmo(\R^n)}.
\]

The first term in \eqref{def_prod} can we written as  $B_1^{(2)}(f,g):=(\mathcal{L}f)g$. We shall see that this term is actually a function in $L^1(\R^n)$. Indeed, let $\lbrace Q_i\rbrace_{i\geq 0}$ be a countable collection of cubes, all of them with fixed sidelength $\ell\leq (4n)^{-1/2}$, independent on $i$, such that it gives a partition of $\R^n$. For all $i\geq 0$, we denote by  $\tilde{Q}_i$ the dilation of the cube $Q_i$ with sidelength $1$. Then we have that
{
\begin{equation}\label{ecu30}
\begin{split}
    \int &\abs{(\Phi\ast f)(x)}\abs{g(x)}\dd x=\sum_{i=0}^\infty\int_{Q_i}\abs{(\Phi\ast f)(x)}\abs{g(x)}\dd x
    \\
    &\leq\sum_{i=0}^\infty\sup_{x\in Q_i}\abs{(\Phi\ast f)(x)}\int_{Q_i}\abs{g(x)}\dd x\leq \sum_{i=0}^\infty\sup_{x\in Q_i}\abs{(\Phi\ast f)(x)}\int_{\tilde{Q}_i}\abs{g(x)}\dd x \\ 
    & %
    \leq \brkt{\sum_{i=0}^\infty\sup_{x\in Q_i}\abs{(\Phi\ast f)(x)}} \norm{g}_{\bmo(\R^n)},
\end{split}
\end{equation}
where the last inequality follows from \eqref{Zipi}. %

Let $\Gamma(x)$ be the truncated nontangential cone
\[
    \Gamma(x)=\lbrace(y,t)\in\R^n\times(0,1/2):\abs{x-y}<t\rbrace,\quad x\in\R^n.
\]
For every $i\geq 0$ and for all $x,\tilde{x}\in Q_i$, it holds that
\[
    \abs{x-\tilde{x}}\leq \sqrt{n}\ell<1/2.
\]
Let $\Psi(x):=\Phi(x/4)4^{-n}$. It follows that for all $x,\tilde{x}\in Q_i$
\[
    \abs{(\Phi\ast f)(\tilde{x})}=\abs{(\Psi_{1/4}\ast f)(\tilde{x})}\leq\sup_{(y,t)\in\Gamma(x)}\abs{(\Psi_{t}\ast f)(y)}.
\] 
Taking supremum on $\tilde{x}$, and integrating on $x\in Q_i$ in both sides of the last inequality, it follows that
\begin{equation*} %
    \sup_{\tilde{x}\in Q_i}\abs{(\Phi\ast f)(\tilde{x})}\lesssim\int_{Q_i}\sup_{(y,t)\in\Gamma(x)}\abs{(\Psi_t\ast f)(y)}\dd x.
\end{equation*}
This inequality, the fact that $\{Q_i\}_{i\geq 0}$ form a partition of $\R^n$ and \eqref{def:norm_h1}, imply that 
\[
    {\sum_{i=0}^\infty\sup_{x\in Q_i}\abs{(\Phi\ast f)(x)}}\lesssim \int\sup_{(y,t)\in\Gamma(x)}\abs{(\Psi_t\ast f)(y)}\dd x\approx\norm{f}_{h^1(\R^n)},
\]
where we are using the independence of the $h^1(\R^n)$ norm on the chosen function $\Psi$.  Finally, using this in \eqref{ecu30}, we obtain that 
\[
    \norm{B_1^{(2)}(f,g)}_{L^1(\R^n)}\lesssim \norm{f}_{h^1(\R^n)}\norm{g}_{\bmo(\R^n)}.
\]
}
To finish the proof, it is enough to define the bilinear operator $B_1:=B_1^{(1)}+B_1^{(2)}$.
\end{proof}
The boundedness of the second term in \eqref{def_prod}, was obtained by applying  the results of  Corollary \ref{Cor:main_bmo} above, as it was expressed as a bilinear Coifman-Meyer multiplier whose symbol is identically one.  However, since by \eqref{boundedness}, $\mathcal{H}f$ defines a function in $H^1(\R^n)$ and $\bmo(\R^n)\subset \BMO(\R^n)$, we could use instead 
\cite{Bon-Gre-Ky}*{Theorem 1.1} to show that 
\[
    (\mathcal{H}f)g=S(\mathcal{H}f,g)+T(\mathcal{H}f,g),
\]
where $S:H^1(\R^n)\times\BMO(\R^n)\rightarrow L^1(\R^n)$ and $T:H^1(\R^n)\times\BMO(\R^n)\rightarrow H^{\log}(\R^n)$. Here $H^{\log}(\R^n)$ denotes the Musielak-Orlicz-Hardy space defined in \cite{Bon-Gre-Ky}, associated  to the Musielak-Orlicz function
    \begin{equation}\label{growth_funct}
        \varphi(x,t):=\frac{t}{\log(e+\abs{x})+\log(e+t)},\quad x\in\R^n,t>0.
    \end{equation}
Using this, one obtains the following result.
\begin{cor}\label{teo_prod_2}
    There exist two continuous bilinear operators on the product space $h^1(\R^n)\times \bmo(\R^n)$,  respectively $B_1:h^1(\R^n)\times \bmo(\R^n)\to L^1(\R^n)$ and $B_2:h^1(\R^n)\times \bmo(\R^n)\to H^{\log}(\R^n)$, such that 
    \[
        fg=B_1(f,g)+B_2(f,g).
    \] 
\end{cor}

\begin{rem}
    Related results to the previous corollary were obtained by  J. Cao, L.D. Ky and D. Tang \cite{Cao-Ky-Yang}*{Theorem 1.1 and Remark 3.1}.   More specifically, a decomposition of the product of functions in $h^1(\R^n)$ and $\bmo(\R^n)$ into two bilinear operators, one of them mapping continuously $h^1(\R^n)\times\bmo(\R^n)$ into $L^1(\R^n)$, and the other one mapping continuously $h^1(\R^n)\times\bmo(\R^n)$ into $h^{\log}(\R^n)$, where $h^{\log}(\R^n)$ is a local Musielak-Orlicz-Hardy space related to the growth function defined in \eqref{growth_funct}. 
    In particular, since the inclusion $H^{\log}(\R^n)\subseteq h^{\log}(\R^n)$ holds, we notice that Corollary \ref{teo_prod_2} above slightly improves  \cite{Cao-Ky-Yang}*{Theorem 1.1 and Remark 3.1}.
\end{rem}

\begin{rem}
    It would be an interesting problem to study the relationship between the potential space $\J_{\Rw}(H^1(\R^n))$ and the Musielak-Orlicz-Hardy space $H^{\log}(\R^n)$ appearing in \cite{Bon-Gre-Ky}, as well as the space $h^{\Phi}_*(\R^n)$ in \cite{Cao-Ky-Yang}*{Theorem 1.1}.  
\end{rem}
The following result can be obtained from either Corollary \ref{teo_prod}, or from Corollary \ref{Cor:main_bmo} by using the embedding $H^1(\R^n)\subseteq h^1(\R^n)$.
\begin{cor}
    There exist two continuous bilinear operators on the product space $H^1(\R^n)\times \bmo(\R^n)$,  respectively $B_1:H^1(\R^n)\times \bmo(\R^n)\to L^1(\R^n)$ and $B_2:H^1(\R^n)\times \bmo(\R^n)\to J_{\Rw_1}( H^1(\R^n))$, such that 
        \[
            fg=B_1(f,g)+B_2(f,g).
        \]
\end{cor}
        
\section{Appendix: The $L^2$ estimates for paraproducts}\label{Appendix}

Although the case $p=2$ is covered by interpolation in Theorem \ref{bmo_Lp}, we shall give here a direct and self-contained proof that relies only on the use of Plancherel's Theorem and quadratic estimates. More precisely, we shall prove the following result.

\begin{prop}  Let $w$ be an admissible  weight satisfying $\inf_{t>0} w(t)>0$ and let $\Rw$ be its regularisation given in Definition \ref{Regularised_weight}. There is a constant $C>0$ such that the estimate
    \begin{equation}\label{eq:L2}
        \norm{\Pi(f,g)}_{\J_{\Rw}(L^2(\R^n))}\leq C\norm{f}_{L^2(\R^n)}\norm{g}_{X_w(\R^n)}
    \end{equation}
    holds for all $f\in L^2(\R^n)$ and $g\in X_w(\R^n)$.
\end{prop}
\begin{proof}
By the decomposition \eqref{eq:decomp}, it is enough to prove the corresponding estimates for both $\Pi_1$ and $\Pi_2$ defined in \eqref{eq:Pi1} and \eqref{eq:Pi2} respectively. 

To this end, let $f\in L^2(\R^n)$, $g\in X_w(\R^n)$ and let  $H=\J_{\Rw^{-1}}h$ with $h\in L^2(\R^n)$.  Note that the assumption on $w$ to be bounded from below, and  \eqref{ecu2sym} implies that for all $\xi\in \R^n$
\[
    \Rw^{-1}(\xi)\lesssim 1,
\]
and thus, by Plancherel Theorem
\begin{equation}\label{eq:boundedness_H}
        \norm{H}_{L^2(\R^n)}\lesssim \norm{h}_{L^2(\R^n)}.
\end{equation}

Note that we can write
\[
    \langle\Pi_1(f,g),H\rangle=\int\!\!\!\!\int_0^\infty (Q_tf)(x)v(t,x)w(t)(Q^{(2)}_tH)(x)\frac{\dd t}{t}\dd x,
\]
where $v(t,x):=(P_t^{(1)}g)(x)m(t)/w(t)$, which, by the definition of $X_w(\R^n)$ and the assumption on $m$ satisfies 
\begin{equation}\label{ecu10}
    \abs{v(t,x)}\lesssim\norm{m}_{L^\infty(\R^n)}\norm{g}_{X_w(\R^n)}, \quad \mbox{for all $t>0$ and $x\in\R^n$.}
\end{equation}
Observe that for fixed $t>0$, on the support of $\widehat{\psi^{(2)}}(t\xi)$, one has that $\abs{\xi}\approx t^{-1}$, and so \eqref{sigma_equiv_weight} and \eqref{equivalence} yield
\begin{equation}\label{eq:pointwise}
    {w(t)}\approx {\Rw(\xi)}.
\end{equation}
Applying the Plancherel theorem,  \eqref{eq:pointwise}, Tonelli's theorem, the properties of $\psi^{(2)}$ and the Plancherel theorem again yield the following quadratic estimate
\begin{equation}\label{eq:quadratic_1}
    \begin{split}
        &{\int\!\!\!\!\int_0^\infty\big|w(t)(Q_t^{(2)}\J_{\Rw^{-1}} h)(x)\big|^2\frac{\dd t}{t}\dd x}\approx {\int\!\!\!\!\int_0^\infty\big|{w(t)\widehat{\psi^{(2)}}(t\xi)\frac{\widehat{h}(\xi)}{\Rw(\xi)}}\big|^2\frac{\dd t}{t}\ddd\xi}\\
        &\qquad\approx {\int\!\!\!\!\int_0^\infty\big|\widehat{\psi^{(2)}}(t\xi)\widehat{h}(\xi)\big|^2\frac{\dd t}{t}\ddd\xi}\lesssim \norm{h}_{L^2(\R^n)}^2.
    \end{split}
\end{equation}
A slighter simpler argument gives
\begin{equation}\label{eq:quadratic_2}
    \bigg[\int\!\!\!\!\int_0^\infty \abs{(Q_tf)(x)}^2\frac{\dd t}{t}\dd x\bigg]^{1/2}\lesssim \norm{f}_{L^2(\R^n)}.
\end{equation}

Finally, using \eqref{ecu10}, the Cauchy-Schwarz inequality,   \eqref{eq:quadratic_1} and \eqref{eq:quadratic_2} yield
\begin{equation}\label{ecudual}
    \begin{split}
   &\abs{\langle\Pi_1(f,g),H\rangle}\lesssim\norm{m}_{L^\infty(\R^n)}\norm{g}_{X_w(\R^n)}\norm{f}_{L^2(\R^n)}\norm{h}_{L^2(\R^n)}.
    \end{split}
\end{equation}

To study $\Pi_2$, we have that
\[
    \langle\Pi_2(f,g),H\rangle=\int\!\!\!\!\int_0^\infty (Q_tf)(x)(Q_t^{(1)}g)(x)(P^{(2)}_tH)(x)m(t)\frac{\dd t}{t}\dd x.
\]
Since $g\in X_w(\R^n)\subset \BMO(\R^n)$,  Theorem \ref{rem_BMO_CM} implies that $|{(Q_t^{(1)}g)(x)}|^2t^{-1}\dd t\dd x$ is a Carleson measure,  with norm bounded by a constant times $\norm{g}_{X_w(\R^n)}^2$.
Thus the Cauchy-Schwarz inequality, the boundedness of $m$, the quadratic estimate \eqref{eq:quadratic_2}, Theorem \ref{Feff-Ste} and \eqref{eq:boundedness_H} yield
\begin{equation}\label{ecu11}
    \begin{split}
    &\abs{\langle\Pi_2(f,g),H\rangle}
   \lesssim\norm{m}_{L^\infty(\R^n)}\norm{f}_{L^2(\R^n)}\norm{g}_{X_w(\R^n)}\norm{h}_{L^2(\R^n)}.
    \end{split}
\end{equation}

Finally, combining \eqref{ecu11} and \eqref{ecudual} we deduce that
\[
\abs{\langle\Pi(f,g),J_{1/\Rw}h\rangle}\lesssim\norm{m}_{L^\infty(\R^n)}\norm{f}_{L^2(\R^n)}\norm{g}_{X_w(\R^n)}\norm{h}_{L^2(\R^n)},
\]
for all $f,h\in L^2(\R^n)$ and $g\in X_w(\R^n)$, which by duality yields \eqref{eq:L2}, and thus finishing the proof of the result.
\end{proof}

\end{document}